\documentclass{gOPT2e}
\def\sr{{}^sr}
\theoremstyle{plain}
\newtheorem{theorem}{Theorem}[section]

\newtheorem{corollary}[theorem]{Corollary}
\newtheorem{lemma}[theorem]{Lemma}
\newtheorem{proposition}[theorem]{Proposition}

\theoremstyle{remark}
\newtheorem{remark}{Remark}

\theoremstyle{definition}

\newcommand{\al}{\alpha}
\newcommand{\be}{\beta}
\newcommand{\ga}{\gamma}
\newcommand{\la}{\lambda}
\newcommand{\de}{\delta}
\newcommand{\eps}{\varepsilon}
\newcommand{\bx}{\bar x}
\newcommand{\by}{\bar y}

\newcommand{\iv}{^{-1} }
\newcommand{\Real}{\mathbb R}
\newcommand {\R} {\mathbb R}

\newcommand {\B} {\mathbb B}
\newcommand {\Sp} {\mathbb S}
\newcommand {\gph} {{\rm gph}\,}

\newcommand {\Er} {{\rm Er}\,}

\newcommand {\sd} {\partial}

\renewcommand{\iff}{$ \Leftrightarrow\ $}
\newcommand{\folgt}{$ \Rightarrow\ $}

\newcommand{\ds}{\displaystyle}

\def\lsc{lower semicontinuous}

\def\RHS{right-hand side}

\newcommand{\norm}[1]{\left\Vert#1\right\Vert}

\newcounter{mycount}
\def\cnta{\setcounter{mycount}{\value{enumi}}}
\def\cntb{\setcounter{enumi}{\value{mycount}}}

\def\smartqedtr{\def\qedtr{\ifmmode\triangle\else{\unskip\nobreak\hfil
\penalty50\hskip1em\null\nobreak\hfil$\triangle$
\parfillskip=0pt\finalhyphendemerits=0\endgraf}\fi}}
\smartqedtr  
\lccode`\-=`\-


\renewcommand {\theenumi} {\roman{enumi}}
\begin{document}

\articletype{Original Article}

\title{{\itshape Error Bounds and Metric Subregularity}}

\author{Alexander Y. Kruger$^*$
\thanks{$^*$Email: a.kruger@federation.edu.au}
\\\vspace{6pt}
{\em{Centre for Informatics and Applied Optimisation,\\
School of Science, Information Technology and Engineering,\\
Federation University Australia,\\
POB 663, Ballarat, Vic, 3350, Australia}}
\\\received{Dedicated to the 40th Anniversary of the journal;\\ its founder and former editor-in-chief, Professor Karl-Heinz Elster;\\ and Professor Alfred G\"opfert, an editorial board member since 1988 in celebration of his 80th birthday}}


\maketitle

\begin{abstract}
Necessary and sufficient criteria for metric subregularity (or calmness) of set-valued mappings between general metric or Banach spaces are treated in the framework of the theory of error bounds for a special family of extended real-valued functions of two variables.
A classification scheme for the general error bound and metric subregularity criteria is presented.
The criteria are formulated in terms of several kinds of primal and subdifferential slopes.

\begin{keywords}
error bounds; slope; metric regularity; metric subregularity; calmness
\end{keywords}

\begin{classcode}
49J52; 49J53; 58C06; 47H04; 54C60
\end{classcode}
\end{abstract}

\section{Introduction}

This paper is another attempt to demonstrate that (necessary and sufficient) criteria for \emph{metric subregularity} (or equivalently \emph{calmness}) of set-valued mappings between general metric or Banach spaces can be treated in the framework of the theory of \emph{error bounds} of extended real-valued functions.
Another objective is to classify the general error bound and subregularity criteria and clarify the relationships between them.

Due to the importance of the three properties mentioned above in both theory and applications, the amount of publications devoted to the properties and corresponding (mostly sufficient) criteria is huge.
The interested reader is referred to the articles by Az{\'e} \cite{Aze03}, Az{\'e} and Corvellec \cite{AzeCor04}, Corvellec and Motreanu \cite{CorMot08}, Gfrerer \cite{Gfr11}, Ioffe \cite{Iof00_}, Ioffe and Outrata \cite{IofOut08}, Ngai and Th{\'e}ra \cite{NgaiThe04,NgaiThe08}, Jong-Shi Pang \cite{Pang97}, Zheng and Ng \cite{ZheNg10,ZheNg12} and the references therein.

Both local and global settings of the properties have proved to be important and have been thoroughly investigated.
In this paper, only local properties are considered.

Let us recall several basic definitions.

An extended-real-valued function $f:X\to\Real_\infty:=\R\cup\{+\infty\}$ on a metric space $X$ is said to have a local \emph{error bound} (cf., e.g., \cite{Aze03,AzeCor04,FabHenKruOut10,Iof79}) with constant $\tau>0$ at a point $\bar{x}\in X$ with $f(\bx)=0$ if there exists a neighbourhood $U$ of $\bx$ such that
\begin{equation}\label{el+}
{\tau}d(x,S(f)) \le f_+(x)\quad \mbox{ for all } x \in U.
\end{equation}
Here $S(f)$ stands for the lower $0$-level set $\{x\in X\mid f(x)\le0\}$.

A set-valued mapping
$F:X\rightrightarrows Y$ is a mapping which
assigns to every $x\in X$ a subset (possibly empty) $F(x)$ of
$Y$.
We use the notation
$$\gph F:=\{(x,y)\in X\times Y\mid
y\in F(x)\}$$
for the graph of $F$ and $F\iv : Y\rightrightarrows
X$ for the inverse of $F$.
This inverse (which always exists) is defined by
$$F\iv(y) :=\{x\in X|\, y\in F(x)\},\quad y\in Y,$$
and satisfies
$$(x,y)\in\gph F \quad\Leftrightarrow\quad (y,x)\in\gph F\iv.$$

A set-valued mapping $F:X\rightrightarrows Y$ between metric spaces is called (locally) \emph{metrically subregular} (cf., e.g., \cite{Mor06.1,RocWet98,DonRoc09,Pen13}) at a point $(\bx,\by)\in\gph F$ with constant $\tau>0$ if there exists neighbourhoods $U$ of $\bx$ and $V$ of $\by$ such that
\begin{equation*}\label{MR0++}
\tau d(x,F^{-1}(\by))\le d(\by,F(x)\cap V) \quad \mbox{ for all } x \in U.
\end{equation*}

A set-valued mapping $F:X\rightrightarrows Y$ between metric spaces is called (locally) \emph{calm} (cf., e.g., \cite{RocWet98,DonRoc09}) at a point $(\bx,\by)\in\gph F$ with constant $\tau>0$ if there exist neighbourhoods $U$ of $\bx$ and $V$ of $\by$ such that
\begin{equation*}
d(y,F(\bx))\le\tau d(x,\bx) \quad \mbox{ for all } x\in U,\, y\in F(x)\cap V.
\end{equation*}

The above two properties represent weaker versions of the more robust \emph{metric regularity} and \emph{Aubin} properties, respectively, which correspond to replacing $\bx$ and $\by$ in the above inequalities by arbitrary (not fixed!) $x\in U$ and $y\in V$; cf. \cite{DonRoc09,Mor06.1,RocWet98}.

An immediate observation is that the calmness of $F$ at $(\bx,\by)$ with  constant $\tau$ is equivalent to the metric subregularity of $F\iv$ at $(\by,\bx)$ with  constant $\tau\iv$; cf. \cite[Theorem~3H.3]{DonRoc09}.
Hence, any metric subregularity criterion automatically translates into a calmness criterion.

Another observation is that neighbourhood $V$ in the original definition of metric subregularity is actually not needed and the definition is equivalent to the existence of a neighbourhood $U$ of $\bx$ such that
\begin{equation}\label{MR0-}
\tau d(x,F^{-1}(\by))\le d(\by,F(x)) \quad \mbox{ for all } x \in U.
\end{equation}
A similar remark can be made regarding the definition of calmness; cf. \cite[Exercise~3H.4]{DonRoc09}.

Comparing inequalities \eqref{el+} and \eqref{MR0-}, one can easily see that metric subregularity of $F$ at $(\bx,\by)$ is equivalent to the local error bound property of the extended real-valued function $x\mapsto d(\by,F(x))$ at $\bx$ (with the same constant).
So one might be tempted to apply the well developed theory of error bounds to characterizing metric subregularity and calmness.
This approach was very successfully followed by Ioffe and Outrata \cite{IofOut08} in finite dimensions.

However, in general the case is not that simple.
Most of the error bound criteria (cf. Section~\ref{S1}) are formulated for \lsc\ functions, but in infinite dimensions the function $x\mapsto d(\by,F(x))$ can fail to be \lsc\ even when $\gph F$ is closed.
As observed by Ngai and Th\'era \cite{NgaiThe08}, in some situations, one can make use of the lower semicontinuous envelope of this function: $x\mapsto\liminf_{u\to x}d(\by,F(u))$, although this breaks the equivalence between error bounds and metric subregularity.

Comparing the criteria for the error bounds and metric subregularity (see Sections~\ref{S1} and \ref{S3}), one can notice that in most cases they look very similarly.
Furthermore, the proofs of these criteria, though formally independent, are usually based on the same ideas.
In fact, when proving regularity or calmness criteria for set-valued mappings, the authors often use error bound-like estimates for an extended real-valued function, but defined on the product space $X\times Y$.
The following function (or a function derived from it):
\begin{gather}\label{f+}
f(x,y):=
\begin{cases}
d(y,\by) & \text{if } (x,y)\in\gph F,\\
+\infty & \text{otherwise}
\end{cases}
\end{gather}
is most commonly used for that purpose; cf. \cite{Iof00_,Iof03,Aze03,AzeCor04,Aze06,Iof07,AzeBeh08, AzeCor09,Iof11,Iof13}.
Observe that this function is \lsc\ if $\gph F$ is closed.

In this paper, the theory of local error bounds in metric or Banach/Asplund spaces is, with little changes in the standard proofs, expanded to a class of extended real-valued functions of two variables including, in particular, functions of the type \eqref{f+}.
Then, metric subregularity criteria for set-valued mappings are formulated as consequences of the corresponding ones for error bounds.

Following the standard trend initiated by Ioffe \cite{Iof00_} (cf. \cite{Aze03,Aze06,AzeCor02,AzeCor04,CorMot08, FabHenKruOut10,FabHenKruOut12, BedKru12,BedKru12.2,NgaKruThe12, AmbGigSav08, DurTroStr13,ApeDurStr13, Iof00_,Iof07,Iof07.2,Iof08.2,Iof13,Iof13.2,Pen10, NgaiThe08,NgaiTronThe13,Ude05,Ude07,Ude10.2, MenYan12}), criteria for error bounds and metric (sub-)regularity of set-valued mappings in metric spaces are formulated in terms of (strong) \emph{slopes} \cite{DMT}.
To simplify the statements in metric and also Banach/Asplund spaces, several other kinds of primal and dual space slopes for real-valued functions and set-valued mappings are introduced in this paper and the relationships between them are established.
These relationships lead to a simple hierarchy of the error bound and metric subregularity criteria.

Some statements in the paper look rather long because each of them contains an almost complete list of criteria applicable in the situation under consideration.
The reader is not expected to read through the whole list.
Instead, they can select a particular criterion or a group of criteria corresponding to the setting of interest to them (e.g., local or nonlocal, in metric or Banach/Asplund spaces, etc.)

Certain important groups of criteria are not considered in the current paper: in terms of linearized objects (directional derivatives and tangent cones of some sort) and limiting objects (subdifferentials, normal cones and coderivatives) as well as criteria for nonlinear, in particular H\"older, error bounds and metric subregularity.
The convex case is only slightly touched on in several statements.

Only general settings are considered.
For metric subregularity and calmness criteria for specific set-valued mappings arising from optimization and variational problems we refer the reader to \cite{HenJouOut02,HenOut01,HenOut05, KlaKum02,KlaKum06,BosJouHen04, Gfr13,Gfr14,GobLop14,Ude10.2,ZheNg10} and the references therein.

Our basic notation is standard, see \cite{Mor06.1,RocWet98,DonRoc09}.
Depending on the context, $X$ and $Y$ are either metric or normed
spaces.
Metrics in all spaces are denoted by the same symbol
$d(\cdot,\cdot)$,
$d(x,A):=\inf_{a\in{A}}\|x-a\|$ is the point-to-set distance from
$x$ to $A$.
$B_\de(x)$ denotes the closed ball with radius $\de$ and centre $x$.
When dealing with product spaces, if not specified otherwise, we assume that the product
topology is given by the maximum type distance/norm.

In the case when $X$ and $Y$ are normed spaces, their topological duals are denoted $X^*$ and $Y^*$, respectively, while $\langle\cdot,\cdot\rangle$ denotes the bilinear form defining the pairing between the spaces.
The closed unit balls in a normed space and its dual are denoted by $\B$ and $\B^*$, respectively, while $\Sp$ and $\Sp^*$ stand for the unit spheres.

We say that a subset $\Omega$ of a metric space is locally closed near $\bar{x}\in\Omega$ if $\Omega\cap{U}$ is closed for some closed neighbourhood $U$ of $\bar{x}$.

Given an $\al\in\R_\infty$, $\alpha_+$ denotes its ``positive'' part:
$\alpha_+:=\max\{\alpha,0\}$.
If $f$ is an extended real-valued function on $X$, then $f_+$ is a function defined, for each $x\in X$, as $f_+(x):=(f(x))_+$.

If $X$ is a normed linear space, $f:X\to\Real_\infty$, $x\in{X}$, and $f(x)<\infty$, then
\begin{gather}
\partial{f}(x) := \left\{x^\ast\in X^\ast\bigl|\bigr.\;
\liminf\limits_{u\to x,\,u\ne x}
\frac{f(u)-f(x)-\langle{x}^\ast,u-x\rangle}
{\norm{u-x}} \ge 0 \right\}\label{Frsd}
\end{gather}
is the \emph{Fr\'echet subdifferential} of $f$ at $x$.
Similarly, if $x\in\Omega\subset X$, then
\begin{gather}
N_{\Omega}(x) := \left\{x^\ast\in X^\ast\bigl|\bigr.\
\limsup_{u\to x,\,u\in\Omega\setminus\{x\}} \frac {\langle x^\ast,u-x\rangle}
{\|u-x\|} \le 0 \right\}\label{Fr}
\end{gather}
is the \emph{Fr\'echet normal cone} to $\Omega$ at $x$.
In the convex case, sets \eqref{Frsd} and \eqref{Fr} reduce to the subdifferential and normal cone in the sense of convex analysis, respectively.
If $f(x)=\infty$ or $x\notin\Omega$, we set, respectively, $\partial{f}(x)=\emptyset$ or $N_{\Omega}(x)=\emptyset$.

If $F:X\rightrightarrows Y$ is a set-valued mapping between normed linear spaces and $(x,y)\in\gph F$, then
 \[
D^{*}F(x,y)(y^{*}):= \{x^{*}\in X^*\mid (x^{*},-y^{*})\in N_{\gph F}(x,y)\}, \ y^{*}\in X^*
 \]
is the {\em Fr\'echet coderivative} of $F$ at $(x,y)$.

The proofs of the main statements rely heavily on two fundamental results of variational analysis: the \emph{Ekeland variational principle} (Ekeland \cite{Eke74}; cf., e.g., \cite[Theorem~2.1]{Kru03.1}, \cite[Theorem 2.26]{Mor06.1}) and  the \emph{fuzzy} (\emph{approximate}) \emph{sum rule} (Fabian \cite{Fab89}; cf., e.g., \cite[Rule~2.2]{Kru03.1}, \cite[Theorem 2.33]{Mor06.1}).
Below we provide these results for completeness.

\begin{lemma}[Ekeland variational principle] \label{l01}
Suppose $X$ is a complete metric space, and $f: X\to\R_\infty$ is lower semicontinuous and bounded from below, $\varepsilon>0, \lambda>0$. If
$$
f(v)<\inf_X f + \varepsilon,
$$
then there exists $x\in X$ such that

(a) $d(x,v)<\lambda $,

(b) $f(x)\le f(v)$,

(c) $f(u)+(\varepsilon/\lambda)d(u,x)\ge f(x)$ for all $u\in X$.
\end{lemma}

\begin{lemma}[Fuzzy sum rule] \label{l02}
Suppose $X$ is Asplund,
$f_1: X \to\R$ is Lipschitz continuous and
$f_2: X \to\R_\infty$
is lower semicontinuous in a neighborhood of $\bar x$ with $f_2(\bar x)<\infty$.
Then, for any $\varepsilon>0$, there exist $x_1,x_2\in X$ with $\|x_i-\bar x\|<\varepsilon$, $|f_i(x_i)-f_i(\bar x)|<\varepsilon$ $(i=1,2)$ such that
$$
\partial (f_1+f_2) (\bar x) \subset \partial f_1(x_1) +\partial f_2(x_2) + \varepsilon\B^\ast.
$$
\end{lemma}

Recall that the (normalized) \emph{duality mapping} \cite[Definition~3.2.6]{Lucc06} $J$ between a normed space $Y$ and its dual $Y^*$ is defined as
\begin{gather}\label{J}
J(y):=\left\{y^*\in \mathbb{S}_{Y^*}\mid \langle y^*,y\rangle=\norm{y}\right\},\; \forall y\in Y.
\end{gather}
Note that $J(-y)=-J(y)$.

The following simple fact of convex analysis is well known (cf., e.g., \cite[Corollary~2.4.16]{Zal02}).

\begin{lemma}\label{ll01}
Let $(Y,\|\cdot\|)$ be a normed space.
\begin{enumerate}
\item
$\sd\|\cdot\|(y)=J(y)$ for any $y\ne0$.
\item
$\sd\|\cdot\|(0)=\B^*$.
\end{enumerate}
\end{lemma}

The structure of the paper is as follows.
In the next section, we present a survey of error bound criteria for extended-real-valued functions on metric and Banach/Asplund spaces.
The criteria are formulated in terms of several kinds of primal and subdifferential slopes.
The relationships between the slopes are presented.
In Section~\ref{S31}, the definitions of the error bound property and slopes are extended to a special family of extended real-valued functions on the product of metric or Banach/Asplund spaces.
The next Section~\ref{EB2} is dedicated to the error bound criteria for functions from this family.
Finally, in Section~\ref{S3}, we demonstrate how the definitions of slopes and error bound criteria from Sections~\ref{S31} and \ref{EB2} translate into the corresponding definitions and criteria for metric subregularity of set-valued mappings.

\section{Error Bounds and Slopes}\label{S1}
In this section, we
recall several error bound criteria in terms of (several kinds of) slopes.

Below
$f:X\to\R_\infty$ is an extended-real-valued function on a metric space, $S(f):=\{x\in X\mid f(x)\le0\}$, and $f(\bx)=0$.

Function $f$ is said to have a local error bound at $\bar{x}$ with constant $\tau>0$ if there exists a neighbourhood $U$ of $\bx$ such that
\begin{equation}\label{el}
{\tau}d(x,S(f)) \le f_+(x)\quad \mbox{ for all } x \in U.
\end{equation}
The \emph{error bound modulus} \cite{FabHenKruOut10} (\emph{conditioning rate} \cite{Pen10,Pen13}):
\begin{equation}\label{rr}
\Er f(\bar{x}) :=\liminf_{\substack{x\to\bar{x},\,f(x)>0}}
\frac{f(x)}{d(x,S(f))}
\end{equation}
coincides with the exact upper bound of all $\tau>0$ such that \eqref{el} holds true for some neighbourhood $U$ of $\bx$ and provides a quantitative characterization of the \emph{error bound property}.

Recall that the local \emph{slope} \cite{DMT} of $f$ at $x$ ($f(x)<\infty$) is defined as
\begin{equation}\label{li}
|\nabla{f}|(x):=\limsup_{u\to{x},\,u\ne x}
\frac{[f(x)-f(u)]_+}{d(u,{x})}.
\end{equation}
In other words,
\begin{equation*}
|\nabla{f}|(x):=\limsup_{u\to{x},\,u\ne x}
\frac{f(x)-f(u)}{d(u,{x})}
\end{equation*}
when $x$ is not a point of local minimum of $f$ and $|\nabla{f}|(x)=0$ otherwise.
This (possibly infinite) quantity provides a convenient primal space characterization of the local behaviour of $f$ near $x$.
If $f(x)=\infty$, we set $|\nabla{f}|(x)=\infty$.

In the original publication \cite{DMT}, constant \eqref{li} was called ``strong slope'' to distinguish it from another (``weak'') construction used in the same article.
As this other construction is not widely used in the theory of error bounds, we do not provide its definition here and omit adjective ``strong'' in the name of constant \eqref{li}.
In \cite{Pen13}, constant \eqref{li} is referred to as \emph{calmness rate} or \emph{downward slope}.
Compare with the \emph{rate of steepest descent} in \cite{Dem10}.

Several modifications of \eqref{li} have been introduced in \cite{FabHenKruOut10} and further developed in \cite{BedKru12,BedKru12.2,FabHenKruOut12,NgaKruThe12,NgaiThe}.
Below we recall some of them which will be used in the rest of the paper.

An important ingredient of definition \eqref{li} (and also definitions (2.3) in \cite{NgaiThe08} and (4) in \cite{NgaiTronThe13}) is the \emph{nonlocal slope} of $f$ at $x$ ($f(x)<\infty$):
\begin{equation}\label{gli0}
|\nabla{f}|^{\diamond}(x):=\sup_{u\ne{x}}
\frac{[f(x)-f_+(u)]_+}{d(u,{x})}.
\end{equation}
Indeed, if $f(x)>0$, then
\begin{equation*}
|\nabla{f}|(x)=\lim_{\eps\downarrow{0}}
|\nabla{f}_{B_\eps(x)}|^{\diamond}(x),
\end{equation*}
where ${f}_{B_\eps(x)}$ is the restriction of $f$ to $B_\eps(x)$.

Note that definition \eqref{gli0} is not absolutely nonlocal.
The supremum in the right-hand side of \eqref{gli0} can be restricted to a certain neighbourhood of $x$ since $f_+$ is bounded from below, and consequently
$[f(x)-f_+(u)]_+/d(u,{x})\to0$ as $d(u,{x})\to\infty$.
This distinguishes \eqref{gli0} from the \emph{least slope} \cite[pp.~127--128]{Sim91} and \emph{global slope} \cite[p.~27]{AmbGigSav08}, \cite[formula (4)]{NgaiThe} where $f(u)$ was used instead of $f_+(u)$ in the corresponding definitions.

If $f$ takes only nonnegative values, then \eqref{gli0} takes a
simpler form:
\begin{equation*}\label{gli1}
|\nabla{f}|^{\diamond}(x):=\sup_{u\ne{x}}
\frac{[f(x)-f(u)]_+}{d(u,{x})}
\end{equation*}
(and coincides with the corresponding definitions in \cite{Sim91,AmbGigSav08,NgaiThe}.)

If $x\ne\bx$, then obviously
\begin{equation}\label{gli0+}
|\nabla{f}|^{\diamond}(x)\ge
\frac{f(x)}{d(x,\bx)}.
\end{equation}

In the sequel, superscript `$\diamond$' (diamond) will be used in all constructions derived from \eqref{gli0} and its analogues to distinguish them from ``conventional'' (local) definitions.

Using \eqref{li} and \eqref{gli0}, we define respectively the \emph{strict outer} \cite{FabHenKruOut10} and \emph{uniform strict outer slopes} \cite{FabHenKruOut12}
of $f$ at $\bx$:
\begin{gather}
\overline{|\nabla{f}|}{}^>(\bar{x}) :=
\liminf_{x\to\bar{x},\;f(x)\downarrow0}
|\nabla{f}|(x),
\label{sl-o}\\
\overline{|\nabla{f}|}{}^{\diamond}(\bar{x}) :=
\liminf_{x\to\bar{x},\;f(x)\downarrow0}
|\nabla{f}|^{\diamond}(x)\label{sl0}
\end{gather}
(with the usual convention that the infimum of the empty set equals $+\infty$).

The word ``strict'' reflects the fact that slopes at nearby points (local or nonlocal) contribute to definitions \eqref{sl-o} and \eqref{sl0} making them analogues of the strict derivative.
The word ``outer'' is used to emphasize that
only points outside the set $S(f)$ are taken into account.
The word ``uniform'' emphasizes the nonlocal (non-limiting) character of $|\nabla{f}|^{\diamond}(x)$ involved in definition \eqref{sl0}.

\begin{remark}\label{RM3}
Definitions \eqref{sl-o} and \eqref{sl0} corresponding to the lower $0$-level set $S(f)$ can be easily extended to the case of the general lower level set $\{x\in X\mid f(x)\le f(\bx)\}$ with an arbitrary finite $f(\bx)$.
It is sufficient to replace $f(x)\downarrow0$ in \eqref{sl-o} and \eqref{sl0} with $f(x)\downarrow f(\bx)$, cf. \cite{FabHenKruOut10,FabHenKruOut12,Ude2,Ude14}.
\end{remark}

\begin{remark}\label{RM4}
One can also consider (smaller) versions of \eqref{sl-o} and \eqref{sl0} corresponding to the one-sided limits $f(x)\downarrow0$ (or more generally $f(x)\downarrow f(\bx)$) in the definitions being replaced by the full ones: $f(x)\to0$ (or $f(x)\to f(\bx)$).
Such an analogue of \eqref{sl-o} is known as the \emph{strict slope} \cite{FabHenKruOut10} (\emph{limiting slope} \cite{Iof00_,DruIofLew,Mor06.2}); compare with the \emph{relaxed slope} \cite{AmbGigSav08} and the \emph{strong relaxed slope} \cite{RosSegSte11}.
\end{remark}

In normed linear spaces, one can use for estimating slopes and hence error bounds some other tools based on either directional derivatives or subdifferentials of some sort.
Below we describe certain tools from the second group.
Some examples of application of directional derivatives for estimating slopes and error bounds can be found, e.g., in \cite{BedKru12,Aze03,Aze06,AzeCor04,Iof03,Ude14}.

Suppose $X$ is a normed linear space.
One can define dual counterparts of the local slopes \eqref{li} and \eqref{sl-o}:
the \emph{subdifferential slope} \cite{FabHenKruOut10} (cf. the \emph{least slope} \cite{Sim91}, the \emph{nonsmooth slope} \cite{Li10}, see also \cite{Pen10,Pen13})
\begin{equation}\label{eg}
|\partial{f}|(x):=\inf_{x^*\in\sd{f}(x)} \|x^*\|
\end{equation}
of $f$ at $x$ ($f(x)<\infty$) and
the \emph{strict outer subdifferential slope} \cite{FabHenKruOut10}
\begin{gather}
\overline{|\sd{f}|}{}^>(\bar{x}):=
\liminf_{x\to\bar{x},\;f(x)\downarrow0}
|\sd{f}|(x)\label{if}
\end{gather}
of $f$ at $\bar{x}$.

Similar to the case of the primal space slopes, one can also define analogues of \eqref{if} as described in Remarks~\ref{RM3} and \ref{RM4} above, cf. \cite{FabHenKruOut10,IofOut08}.

The next proposition summarizes the relationships between the slopes.

\begin{proposition}\label{nc}
\begin{enumerate}
\item
If $0<f(x)<\infty$, then
$|\nabla{f}|(x)\le|\nabla{f}|^{\diamond}(x)$;
\item
$\overline{|\nabla{f}|}{}^>(\bar{x}) \le
\overline{|\nabla{f}|}{}^{\diamond}(\bar{x})$;
\item
$\overline{|\nabla{f}|}{}^{\diamond}(\bar{x})\ge
\ds\liminf_{x\to\bar{x},\;f(x)\downarrow0}
\frac{f(x)}{d(x,\bx)}$.
\cnta
\end{enumerate}
Suppose $X$ is a normed linear space.
\begin{enumerate}
\cntb
\item
$|\nabla{f}|(x)\le|\partial{f}|(x)$ for all $x\in X$ with $f(x)<\infty$;
\item
$\overline{|\nabla{f}|}{}^>(\bar{x})
\le\overline{|\partial{f}|}{}^>(\bar{x})$;
\item
if $X$ is Asplund and $f_+$ is lower semicontinuous near $\bar{x}$, then\\
$\overline{|\nabla{f}|}{}^>(\bar{x})
=\overline{|\partial{f}|}{}^>(\bar{x})$; \item
if $f$ is convex, then
$|\nabla{f}|(x)=|\sd{f}|(x)$ for all $x\in X$ with $f(x)<\infty$ and
$\overline{|\nabla{f}|}{}^{\diamond}(\bar{x})
=\overline{|\nabla{f}|}{}^>(\bar{x})
=\overline{|\nabla{f}|}{}^{>+}(\bar{x})
=\overline{|\partial{f}|}{}^>(\bar{x})
=\overline{|\partial{f}|}{}^{+>}(\bar{x})$.
\end{enumerate}
\end{proposition}

Parts (i), (ii), (iv), and (v) of Proposition~\ref{nc} follow directly from the definitions, see also \cite{Aze06,Iof00_}.
Part (iii) is a consequence of \eqref{gli0+}.
Part (vi) was proved in \cite[Proposition~5(ii)]{FabHenKruOut10} using the Ekeland variational principle (Lemma~\ref{l01}), cf. \cite[Lemma 2.1]{MenYan12}, \cite[Remark~3.2]{Ude2}.
The first equality in (vii) can be found in numerous publications,
cf. \cite{AzeCor02,AzeCor04,Ude05,AmbGigSav08,CorMot08,Iof08.2, Ude10.2,Iof13.2}.
For the other equalities in (vii), see \cite[Theorem~5]{FabHenKruOut10}.
Note that in most publications cited above, $X$ is assumed a Banach space and $f$ \lsc, but these additional assumptions seem to be superfluous.

The uniform strict slope \eqref{sl0} provides the necessary and sufficient characterization of error bounds, cf. \cite[Theorem~1]{FabHenKruOut12}.

\begin{theorem}\label{ a}
\begin{enumerate}
\item [(i)]\label{ a.1}
$\displaystyle{\rm Er}\,{f}(\bar{x}) \le
\overline{|\nabla{f}|}{}^{\diamond}(\bar{x})$;
\item [(ii)]\label{ a.2}
if $X$ is a Banach space and $f_+$ is lower semicontinuous near $\bar{x}$, then
$\displaystyle{\rm Er}\,{f}(\bar{x}) =
\overline{|\nabla{f}|}{}^{\diamond}(\bar{x})$.
\end{enumerate}
\end{theorem}

\begin{remark}
Analyzing the proof of \cite[Theorem~1]{FabHenKruOut12} (or more general Theorem~\ref{3T1} in Section~\ref{EB2}), one can see that Theorem~\ref{ a} remains true if the nonlocal slope \eqref{gli0} is replaced in definition \eqref{sl0} of the uniform strict slope by a smaller ``restricted'' nonlocal slope
\begin{equation}\label{ler}
|\nabla{f}|^{\diamond}_{D(x)}(x):=\sup_{u\in D(x),\,u\ne{x}}
\frac{[f(x)-f_+(u)]_+}{d(u,{x})},
\end{equation}
where $D(x)$ is any subset of $X$ containing $S(f)$.
For instance, one can take $D(x)=\{u\in X\mid d(u,S(f))\le d(x,S(f))\}$.
In this case (and under the natural assumption that $f(x)>0$), \eqref{ler} reduces to the \emph{subslope} of $f$ at $x$ introduced in \cite{ChaChe14}.
Another obvious possibility is to take $D(x)=S(f)$ in which case \eqref{ler} becomes
\begin{equation*}
|\nabla{f}|^{\diamond}_{S(f)}(x)=
\frac{f_+(x)}{d(x,S(f))}
\end{equation*}
(with the convention $0/0=0$).
Substituting this quantity into \eqref{sl0} instead of $|\nabla{f}|^{\diamond}(x)$ makes \eqref{sl0} trivially equal to ${\rm Er}\,{f}(\bar{x})$.
\end{remark}

Thanks to Theorem~\ref{ a} and Proposition~\ref{nc}, one can formulate several quantitative and qualitative criteria of the error bound property in terms of various slopes.

\begin{corollary}\label{_a.1}
Let $\ga>0$.
Consider the following conditions:
\renewcommand {\theenumi} {\alph{enumi}}
\begin{enumerate}
\item
$f$ has a local error bound at $\bx$ with constant $\tau>0$;
\item\label{9b}
$\overline{|\nabla{f}|}{}^{\diamond}(\bar{x})>\ga$,\\ i.e., for some $\rho>0$ and any $x\in B_\rho(\bx)$ with $0<f(x)<\rho$, it holds $|\nabla{f}|^{\diamond}(x)>\ga$, and consequently, there is a $u\in X$ such that
\begin{gather*}\label{_a.11}
f(x)-f_+(u)>\ga d(u,x);
\end{gather*}
\item
$\ds\liminf_{x\to\bar{x},\;f(x)\downarrow0}
\frac{f(x)}{d(x,\bx)}>\ga$; \item\label{9c}
$\overline{|\nabla{f}|}{}^{>}(\bar{x})>\ga$,\\ i.e., for some $\rho>0$ and any $x\in B_\rho(\bx)$ with $0<f(x)<\rho$, it holds $|\nabla{f}|(x)>\ga$, and consequently, for any $\eps>0$, there is a $u\in B_\eps(x)$ such that
\begin{gather}\label{cor9}
f(x)-f(u)>\ga d(u,x);
\end{gather}
\item
$\ds\liminf_{x\to\bar{x},\;f(x)\downarrow0}
\max\left\{|\nabla{f}|(x), \frac{f(x)}{d(x,\bx)}\right\}>\ga$,\\
i.e., for some $\rho>0$ and any $x\in B_\rho(\bx)$ with $0<f(x)<\rho$ and $f(x)/d(x,\bx)\le\ga$, it holds $|\nabla{f}|(x)>\ga$, and consequently, for any $\eps>0$, there is a $u\in B_\eps(x)$ such that \eqref{cor9} holds true;
\item\label{9d}
$X$ is a normed space and $\overline{|\sd{f}|}{}^{>}(\bar{x})>\ga$,\\ i.e., for some $\rho>0$ and any $x\in B_\rho(\bx)$ with $0<f(x)<\rho$, it holds $|\sd{f}|(x)>\ga$, and consequently $\|x^*\|>\ga$ for all $x^*\in\sd f(x)$;
\item
$\ds\liminf_{x\to\bar{x},\;f(x)\downarrow0}
\max\left\{|\sd{f}|(x), \frac{f(x)}{\|x-\bx\|}\right\}>\ga$,\\
i.e., for some $\rho>0$ and any $x\in B_\rho(\bx)$ with $0<f(x)<\rho$ and $f(x)/\|x-\bx\|\le\ga$, it holds $|\sd{f}|(x)>\ga$, and consequently $\|x^*\|>\ga$ for all $x^*\in\sd f(x)$.
\end{enumerate}
\renewcommand {\theenumi} {\roman{enumi}}
The following implications hold true:
\begin{enumerate}
\item
{\rm (c) \folgt (e)};
\item
{\rm \eqref{9c} \folgt (e)};
\item
{\rm (e) \folgt \eqref{9b}};
\item
if $\ga<\tau$, then {\rm (a) \folgt \eqref{9b}};
\item
if $X$ is a normed space, then \eqref{9c}~\folgt \eqref{9d} and {\rm (e)~\folgt (g)}.
\cnta
\end{enumerate}
Suppose $X$ is complete and $f_+$ is lower semicontinuous near $\bar{x}$. Then,
\begin{enumerate}
\cntb
\item
if $\tau\le\ga$, then {\rm \eqref{9b}~\folgt (a)}.
\cnta
\end{enumerate}
Suppose, additionally, that $X$ is a Banach space. Then,
\begin{enumerate}
\cntb
\item
if $X$ is Asplund, then \eqref{9c}~\iff \eqref{9d} and {\rm (e)~\iff (g)};
\item
if $f$ is convex, then \eqref{9b}~\iff \eqref{9c}~\iff \eqref{9d}.
\end{enumerate}
\end{corollary}

Criterion \eqref{9b} in the above proposition is a version of \cite[Basic Lemma]{Iof00_}; see also, \cite[Theorem 2(ii)]{Ham94}, \cite[Theorem~1]{WuYe02}, \cite[Theorem~3.1]{Wu03}, \cite[Corollary~2.3]{NgaiThe08}, \cite[Corollary~4.3]{KlaKruKum12}, \cite[Remark~6.2.2]{GobLop14}.

Criteria \eqref{9c} and \eqref{9d} can be found, e.g., in \cite[Theorem~2.1]{IofOut08}; see also \cite[Theorem~1]{Iof79}, \cite[Theorem~3.1]{Ye98}, \cite[Theorem~3.1]{LedZhu99}, \cite[Theorem~2.4]{Jou00}, \cite[Theorem~5.2]{AzeCor04}, \cite[Corollary~3.1 and Theorem~3.2]{NgaiThe04}, \cite[Corollary~2]{NgaiThe09}, \cite[Theorem~4.12]{Pen10}, \cite[(1.8)]{ZheNg10}, \cite[Proposition~2.1]{MenYan12}, \cite[(R1)]{ZheWei12}, \cite[Corollary~4.5]{ApeDurStr13}, \cite[Corollary~1]{NgaiTronThe13}, \cite[Theorem~3.2]{WeiHe14}, \cite[Corollary~4.1]{ZheHe14}).

Criterion (e) is a combination of criteria (c) and (d), while criterion (g) is a combination of criteria (c) and (f).

The equivalence of (a) and (f) in the convex case can be found, e.g., in \cite[Theorem~2.5]{WuYe03}, \cite[(R1) and (R2)]{ZheWei12}.

\begin{corollary}\label{C1}
Suppose $X$ is complete and $f_+$ is lower semicontinuous near $\bar{x}$. Then,
$f$ has a local error bound at $\bx$ provided that one of the following conditions holds true:
\renewcommand {\theenumi} {\alph{enumi}}
\begin{enumerate}
\item
$\overline{|\nabla{f}|}{}^{\diamond}(\bar{x})>0$; \item
$\ds\liminf_{x\to\bar{x},\;f(x)\downarrow0}
\frac{f(x)}{d(x,\bx)}>0$; \item
$\overline{|\nabla{f}|}{}^{>}(\bar{x})>0$;
\item
$\ds\liminf_{x\to\bar{x},\;\frac{f(x)}{d(x,\bx)}\downarrow0}
|\nabla{f}|(x)>0$;
\item
$X$ is an Asplund space and $\overline{|\sd{f}|}{}^{>}(\bar{x})>0$; \item
$X$ is an Asplund space and $\ds\liminf_{x\to\bar{x},\;\frac{f(x)}{\|x-\bx\|}\downarrow0}
|\sd{f}|(x)>0$. \end{enumerate}
\renewcommand {\theenumi} {\roman{enumi}}
Moreover,
\begin{enumerate}
\item
condition {\rm (a)} is also necessary for the local error bound property of $f$ at $\bx$;
\item
{\rm (b) \folgt (d)};
\item
{\rm (c) \folgt (d)};
\item
{\rm (d) \folgt (a)};
\item
{\rm (e) \folgt (f)};
\item
if $X$ is Asplund, then {\rm (e) \iff (c)} and {\rm (f) \iff (d)}.
\end{enumerate}
\end{corollary}

\begin{remark}
Conditions (b)--(f) in Corollary~\ref{C1} are not necessary.
They can fail for \lsc\ and even Lipschitz continuous functions on $\R$ possessing the error bound property, cf. \cite[Examples~7 and 8]{FabHenKruOut10}.
\end{remark}

Criterion (f) can be found in \cite{Gfr11}.

\begin{remark}
One of the main tools in the proof of inequality
$$\overline{|\nabla{f}|}{}^>(\bar{x})
\ge\overline{|\partial{f}|}{}^>(\bar{x})$$
in Proposition~\ref{nc}(vi) which is crucial for the sufficient error bound criterion in Corollary~\ref{C1}(e) is the fuzzy sum rule (Lemma~\ref{l02}) for Fr\'echet subdifferentials in Asplund spaces.
The inequality and the corresponding sufficient criterion can be extended to general Banach spaces.
For that, one has to replace Fr\'echet subdifferentials with some other (possibly abstract) subdifferentials on the given space satisfying a certain set of natural properties including a kind of sum rule (\emph{trustworthy} subdifferentials \cite{Iof83,Iof00_}), cf. \cite[Proposition~1.13]{Aze03}, \cite[Proposition~2.3]{Aze06}, \cite[Proposition~4.1]{AzeCor04}, \cite[Proposition~6]{FabHenKruOut10},
e.g., \emph{Ioffe approximate} or \emph{Clarke} subdifferentials.
Note that the opposite inequality guaranteed by Proposition~\ref{nc}(vi) is specific for Fr\'echet subdifferentials and cannot be extended beyond Asplund spaces unless $f$ is convex near $\bx$, cf. Proposition~\ref{nc}(vii).
\end{remark}

\begin{remark}
The seemingly more general case of nonlinear error bounds, i.e., when the linear estimate \eqref{el} is replaced by the inequality
\begin{equation*}
d(x,S(f)) \le\varphi(f_+(x))\quad \mbox{ for all } x \in U,
\end{equation*}
where $\varphi:\R_+\to\R_+$ is a given function, can be handled within the framework of the discussed above ``linear'' theory for the composite function $x\mapsto\varphi(f_+(x))$.
To apply the criteria in Corollaries~\ref{_a.1} and \ref{C1}, one needs to compute the slopes of this function, i.e., use some calculus of slopes which is pretty simple, e.g., in the typical case of H\"older-type estimates $x\mapsto(f_+(x))^q$ with $q\in(0,1)$.
\end{remark}

\section{Error Bounds and Slopes for Functions of Two Variables}\label{S31}
Now we consider a more general model when function $f$ depends on two variables: $f:X\times Y\to\R_\infty$.
Both $X$ and $Y$ are metric spaces.


We assume that $f(\bx,\by)=0$, and $f$ depends on its second variable in a special way:
\begin{itemize}
\item [(P1)]
$f(x,y)>0$ if $y\ne\by$,
\item [(P2)]
$\ds\liminf_{f(x,y)\downarrow0}\frac{f(x,y)}{d(y,\by)}>0$.
\end{itemize}
In particular, $f(x,y)\downarrow0\;\Rightarrow\;y\to\by$.

Observe that the case of a function $f:X\to\R_\infty$ of a single variable can be covered by considering its extension $\tilde f:X\times Y\to\R_\infty$ defined, for some $\by\in Y$, by
\begin{equation}\label{fti}
\tilde f(x,y)=
\begin{cases}
f(x) &\mbox{if } y=\by,
\\
\infty &\mbox{otherwise}.
\end{cases}
\end{equation}
Conditions (P1) and (P2) are obviously satisfied.

\subsection{Error bounds}
We are interested in a special kind of error bounds of $f$ with respect to the first argument.

We say that $f$ has an \emph{error bound} with respect to $x$ at $(\bx,\by)$ with constant $\tau>0$ if there exists a neighbourhood $U$ of $\bx$ such that
\begin{equation}\label{eb}
\tau d(x,S(f)) \le f_+(x,y)\quad \mbox{for all } x\in U,\; y\in Y,
\end{equation}
where
$S(f):=\{x\in X|\ f(x,y)\le0 \mbox{ for some } y\in Y\}$.
In view of (P1),
$$S(f)=\{x\in X|\ f(x,\by)\le0\}.$$

Of course, \eqref{eb} is equivalent to
\begin{equation*}
\tau d(x,S(f)) \le\inf_{y\in Y} f(x,y)\quad \mbox{for all } x\in U,
\end{equation*}
which is the usual error bound property for the function $x\mapsto\inf_{y\in Y} f(x,y)$, but for the goals of the current paper, it is more appropriate to use the setting of \eqref{eb}.

The error bound property \eqref{eb} can be equivalently characterized using the following modification of \eqref{rr}:
\begin{gather}\label{rr2}
\Er f(\bx,\by):=
\liminf_{\substack{x\to\bx\\f(x,y)>0}} \frac{f(x,y)}{d(x,S(f))}.
\end{gather}

It is easy to see that, in the special case of function $\tilde f$ defined by \eqref{fti}, definition \eqref{rr2} reduces to \eqref{rr}.

Note that definition \eqref{rr2} (as well as the error bound property defined by \eqref{eb}) looks local only in $x$.
In fact, thanks to (P2), it is local in both $x$ and $y$.
Indeed, it admits the following equivalent representations.
\begin{proposition}\label{pr-f}
$\ds\Er f(\bx,\by)=
\liminf_{\substack{x\to\bx,\,y\to\by\\f(x,y)>0}} \frac{f(x,y)}{d(x,S(f))}=
\liminf_{\substack{x\to\bx,\,f(x,y)\downarrow0}} \frac{f(x,y)}{d(x,S(f))}$.
\end{proposition}
\begin{proof}
The inequalities
\begin{gather*}\label{pr-f-3}
\Er f(\bx,\by)\le
\liminf_{\substack{x\to\bx,\,y\to\by\\f(x,y)>0}} \frac{f(x,y)}{d(x,S(f))}\le
\liminf_{\substack{x\to\bx,\,f(x,y)\downarrow0}} \frac{f(x,y)}{d(x,S(f))}
\end{gather*}
follow from (P2) and the obvious implications:
$$
f(x,y)\downarrow0
\quad\Rightarrow\quad
y\to\by,\,f(x,y)>0
\quad\Rightarrow\quad
f(x,y)>0.
$$
If $\Er f(\bx,\by)=\infty$, then the claimed equalities hold trivially.
If $\Er f(\bx,\by)<\ga<\infty$, then there exists a sequence $(x_k,y_k)\in X\times Y$ with $f(x_k,y_k)>0$ such that $x_k\to\bx$ as $k\to\infty$ and $f(x_k,y_k)/d(x_k,S(f))<\ga$, $k=1,2,\ldots$.
Hence, $d(x_k,S(f))\to0$ and consequently $f(x_k,y_k)\downarrow0$ as $k\to\infty$.
It follows that
\begin{gather*}
\liminf_{\substack{x\to\bx,\,f(x,y)\downarrow0}} \frac{f(x,y)}{d(x,S(f))}\le \liminf_{k\to\infty}\frac{f(x_k,y_k)}{d(x_k,S(f))}\le\ga
\end{gather*}
and consequently
\begin{gather*}
\liminf_{\substack{x\to\bx,\,f(x,y)\downarrow0}} \frac{f(x,y)}{d(x,S(f))}\le \Er f(\bx,\by).
\end{gather*}
\end{proof}

\subsection{Nonlocal slopes}
The roles of variables $x$ and $y$ in definitions \eqref{eb} and \eqref{rr2} are different.
To better reflect this,
we are going to consider the following asymmetric maximum-type distance in $X\times Y$ depending on a positive parameter $\rho$:
\begin{equation}\label{drho}
d_\rho((x,y),(u,v)):=\max\{d(x,u),\rho d(y,v)\}.
\end{equation}
This is a pretty common trick, e.g., when studying regularity properties of set-valued mappings, cf. \cite{Aze06,AzeCor04,AzeBeh08,AzeCor09, Iof07,Iof10,Iof11}.
Alternatively, one can use the parametric sum-type metric (cf. \cite{Iof00_,Iof03}):
\begin{gather}\label{drho1}
d_\rho^1((x,y),(u,v)):=d(x,u)+\rho d(y,v).
\end{gather}

To formulate (nonlocal) primal space characterizations of the error bound property \eqref{eb}, we are going to use the following modifications of slopes \eqref{gli0} and \eqref{sl0}:
\begin{gather}\label{nls-f}
|\nabla{f}|_{\rho}^{\diamond}(x,y):=
\sup_{(u,v)\ne(x,y)}
\frac{[f(x,y)-f_+(u,v)]_+}{d_\rho((x,y),(u,v))},
\\\label{uss-f}
\overline{|\nabla{f}|}{}^{\diamond}(\bar{x},\by):=
\lim_{\rho\downarrow0}
\inf_{\substack{d(x,\bx)<\rho,\,0<f(x,y)<\rho}}\,
|\nabla{f}|{}^{\diamond}_{\rho}(x,y),
\end{gather}
which will be called, respectively,
the \emph{nonlocal $\rho$-slope} of $f$ at $(x,y)$
and
the \emph{uniform strict slope}.
It is assumed in \eqref{nls-f} that
$f(x,y)<\infty$.

Definition \eqref{nls-f} of the nonlocal $\rho$-slope is a realization of definition \eqref{gli0} for the case of a function on a product space with the product metric defined by \eqref{drho}.
In definition \eqref{uss-f}, we have not only $x\to\bx$ and $f(x,y)\downarrow0$, but also the metric on $X\times Y$ used in the definition of the nonlocal $\rho$-slope $|\nabla{f}|{}^{\diamond}_{\rho}(x,y)$ changing with the contribution of the $y$ component diminishing as $\rho\downarrow0$.

\subsection{Local slopes}
The local analogues of \eqref{nls-f} and \eqref{uss-f} are defined as follows:
\begin{gather}\label{ls-f}
|\nabla{f}|_{\rho}(x,y):=
\limsup_{\substack{u\to x,\,v\to y\\
(u,v)\ne(x,y)}}
\frac{[f(x,y)-f(u,v)]_+}{d_\rho((u,v),(x,y))},
\\\label{ss-f}
\overline{|\nabla{f}|}{}^{>}(\bar{x},\by):=
\lim_{\rho\downarrow0}
\inf_{\substack{d(x,\bx)<\rho,\,0<f(x,y)<\rho}}\,
|\nabla{f}|_{\rho}(x,y)
\end{gather}
and are called, respectively, the \emph{$\rho$-slope} of $f$ at $(x,y)$ ($f(x,y)<\infty$) and the \emph{strict outer slope} of $f$ at $(\bx,\by)$.

Definition \eqref{ls-f} of the $\rho$-slope is a realization of definition \eqref{li} for the case of a function on a product space with the product metric defined by \eqref{drho}, cf. \cite{AzeBeh08,Iof03,Iof07,Iof13}.

\begin{proposition}\label{pr2}
\begin{enumerate}
\item
$|\nabla{f}|_{\rho}(x,y)\le
|\nabla{f}|_{\rho}^{\diamond}(x,y)$\\
for all $\rho>0$ and all $(x,y)\in X\times Y$ with $0<f(x,y)<\infty$;
\item
$\overline{|\nabla{f}|}{}^{>}(\bar{x},\by)\le
\overline{|\nabla{f}|}{}^{\diamond}(\bx,\by)$;
\item
$\overline{|\nabla{f}|}{}^{\diamond}(\bar{x})\ge
\ds\liminf_{x\to\bar{x},\;f(x,y)\downarrow0}
\frac{f(x,y)}{d(x,\bx)}$.
\end{enumerate}
\end{proposition}

\begin{proof}
(i) and (ii) follow from comparing definitions \eqref{nls-f}, \eqref{uss-f}, \eqref{ls-f}, and \eqref{ss-f}.

(iii)
Let $\overline{|\nabla{f}|}{}^{\diamond}(\bar{x},\by) <\ga<\infty$ and $\rho>0$.
By (P2), one can find a $\rho'\in(0,\rho)$ such that \begin{align}\label{ndi}
\frac{f(x,y)}{d(y,\by)}>\rho'\ga
\end{align}
as long as $0<f(x,y)<\rho'$.
By \eqref{uss-f}, there exists a point $(x,y)\in X\times Y$ with $d(x,\bx)<\rho'$ and $0<f(x,y)<\rho'$ such that $|\nabla{f}|{}^{\diamond}_{\rho'}(x,y)<\ga$, i.e., by \eqref{nls-f},
$$
\frac{f(x,y)-f_+(u,v)}{d_{\rho'}((x,y),(u,v))}<\ga
$$
for all $(u,v)\ne(x,y)$.
Observe that $(x,y)\ne(\bx,\by)$ since $f(x,y)>0$.
Hence,
$$
\frac{f(x,y)}{d_{\rho'}((x,y),(\bx,\by))} =\min\left\{\frac{f(x,y)}{d(x,\bx)}, (\rho')\iv\frac{f(x,y)}{d(y,\by)}\right\}<\ga.
$$
Together with \eqref{ndi}, this implies
$$
\frac{f(x,y)}{d(x,\bx)}<\ga
$$
and consequently,
\begin{align*}
\inf_{\substack{d(x,\bx)<\rho,\,0<f(x,y)<\rho}}\,
\frac{f(x,y)}{d(x,\bx)}<\ga.
\end{align*}
Taking limits as $\rho\downarrow0$ and $\ga\downarrow \overline{|\nabla{f}|}{}^{\diamond}(\bar{x},\by)$, we arrive at the claimed inequality.
\end{proof}

\subsection{Subdifferential slopes}
If $X$ and $Y$ are normed linear spaces, one can define subdifferential counterparts of the local slopes \eqref{ls-f} and \eqref{ss-f}.
In the product space $X\times Y$, along with the usual $l_\infty$-type norm
\begin{gather*}
\|(u,v)\|=\max\{\|u\|,\|v\|\},\quad (u,v)\in X\times Y,
\end{gather*}
we are going to consider the $\rho$-norm $\|\cdot\|_\rho$ being the realization of the $\rho$-metric \eqref{drho}:
\begin{gather*}\label{nrho}
\|(u,v)\|_\rho=\max\{\|u\|,\rho\|v\|\},\quad (u,v)\in X\times Y.
\end{gather*}
The corresponding dual norm (we keep the same notation $\|\cdot\|_\rho$ for it) is of the form:
\begin{gather}\label{nrho*}
\|(u^*,v^*)\|_\rho=\|u^*\|+\rho\iv\|v^*\|,\quad (u^*,v^*)\in X^*\times Y^*.
\end{gather}

The subdifferential slopes are defined as follows:
\begin{gather}\label{sds-f}
|\sd{f}|_{\rho}(x,y):=
\inf_{\substack{(x^*,y^*)\in\sd f(x,y),\,
\|y^*\|<\rho}} \|x^*\|,
\\
\overline{|\sd{f}|}{}^{>}(\bar{x},\by):=
\lim_{\rho\downarrow0}
\inf_{\substack{d(x,\bx)<\rho,\,0<f(x,y)<\rho}}\,
|\sd{f}|_{\rho}(x,y),
\label{ssds-f}
\end{gather}
and called, respectively,
the \emph{subdifferential $\rho$-slope} of $f$ at $(x,y)$ ($f(x,y)<\infty$) and the \emph{strict outer subdifferential slope} of $f$ at
$(\bx,\by)$.

\begin{theorem}\label{th-f}
\begin{enumerate}
\item
$|\nabla{f}|_{\rho}(x,y)\le
|\sd{f}|_{\rho^2}(x,y)+\rho$ for all $\rho>0$ and all $(x,y)\in X\times Y$ with $f(x,y)<\infty$;
\item
$\overline{|\nabla{f}|}{}^{>}(\bar{x},\by)\le
\overline{|\sd{f}|}{}^{>}(\bx,\by)$;
\item
if $X$ and $Y$ are Asplund and $f_+$ is \lsc\ near $(\bar{x},\by)$, then
$\overline{|\nabla{f}|}{}^{>}(\bar{x},\by)=
\overline{|\sd{f}|}{}^{>}(\bx,\by)$.
\end{enumerate}
\end{theorem}

\begin{proof}
(i) Let $f(x,y)<\infty$, $\rho>0$, $(x^*,y^*)\in\sd f(x,y)$, and $\|y^*\|<\rho^2$.
By definition \eqref{Frsd} of the Fr\'echet subdifferential and taking into account that the Fr\'echet subdifferential is invariant to renorming of a space, we have
$$
\liminf_{\substack{(u,v)\to(x,y)\\
(u,v)\ne(x,y)}}
\frac{f(x,y)-f(u,v)-\langle(x^*,y^*),(u,v)-(x,y)\rangle} {\|(u,v)-(x,y)\|_\rho}\ge0.
$$
It follows that
$$
\limsup_{\substack{(u,v)\to(x,y)\\
(u,v)\ne(x,y)}}
\frac{f(u,v)-f(x,y)} {\|(u,v)-(x,y)\|_\rho}
\le\|(x^*,y^*)\|_\rho
\le\|x^*\|+\rho.
$$
Comparing the first expression with definition \eqref{ls-f} and taking into account that the last expression is positive, we conclude that $|\nabla{f}|_{\rho}(x,y)\le\|x^*\|+\rho$.
The assertion follows after taking infimum in the \RHS\ of the last inequality over all
$(x^*,y^*)\in\sd f(x,y)$ with $\|y^*\|<\rho^2$.

(ii) follows from (i) due to representations \eqref{ss-f}, \eqref{ssds-f}, and the simple observation:
$$
\{(x,y)
|\ d(x,\bx)<\rho^2,\,f(x,y)<\rho^2\}\subset \{(x,y)
|\ d(x,\bx)<\rho,\,f(x,y)<\rho\}
$$
when $\rho\in(0,1)$.

(iii)
Let $X$ and $Y$ be Asplund and $f_+$ be \lsc\ near $(\bar{x},\by)$ (in the product topology).
Thanks to (ii), we only need to prove that
$\overline{|\nabla{f}|}{}^>(\bar{x},\by)\ge
\overline{|\sd{f}|}{}^>(\bx,\by)$.
If $\overline{|\nabla{f}|}{}^>(\bar{x},\by)=\infty$, the assertion is trivial.
Let $\overline{|\nabla{f}|}{}^>(\bar{x},\by)<\gamma<\infty$.
Choose a $\ga'\in(\overline{|\nabla{f}|}{}^>(\bar{x},\by),\gamma)$ and an arbitrary $\rho>0$.
Set $\rho'=\min\{1,\ga\iv\}\rho$.
By definitions \eqref{ss-f} and \eqref{ls-f},
one can find a point $(x,y)\in X\times Y$ such that $d(x,\bx)<\rho'$, $0<f(x,y)<\rho'$, $f$ is \lsc\ near $(x,y)$, and
$$
f(x,y)-f(u,v)\le\ga'\|(u,v)-(x,y)\|_{\rho'}\quad \mbox{for all } (u,v) \mbox{ near } (x,y).
$$
In other words, $(x,y)$ is a point of local minimum of the function $$(u,v)\mapsto f(u,v)+\ga'\|(u,v)-(x,y)\|_{\rho'}.$$
Take an
$$\eps\in(0,\min\{\rho-d(x,\bx),\rho-f(x,y),\ga-\ga'\})$$
sufficiently small such that $f$ is \lsc\ on $B_\eps((x,y))$ and $B_\eps(x)\cap S(f)=\emptyset$.
Applying the \emph{fuzzy sum rule} (see, e.g., \cite[Theorem 2.33]{Mor06.1}), we find points $(z,w)\in X\times Y$ and $(x^*,y^*)\in\sd f(z,w)$ such that $d((z,w),(x,y))<\eps$, $f(z,w)<f(x,y)+\eps$, and $\|(x^*,y^*)\|_{\rho'}<\ga'+\eps$.
It follows that
$d(z,\bx)<\rho$, $0<f(z,w)<\rho$, $\|x^*\|<\ga$, and $\|y^*\|<{\rho'}\ga\le\rho$.
Hence, $|\sd{f}|_{\rho}(z,w)<\ga$ and consequently $\overline{|\sd{f}|}{}^>(\bar{x},\by)\le\ga$.
The claimed inequality follows after letting $\ga\to\overline{|\nabla{f}|}{}^>(\bar{x},\by)$.
\end{proof}

\begin{remark}
The subdifferential $\rho$-slope \eqref{sds-f} of $f$ at $(x,y)$ can be replaced in definition \eqref{ssds-f} by the following modification:
\begin{gather}\label{sds-f'}
|\sd{f}|_{\rho}'(x,y):=
\inf_{(x^*,y^*)\in\sd f(x,y)}
\|(x^*,y^*)\|_\rho,
\end{gather}
where norm $\|\cdot\|_\rho$ is given by \eqref{nrho*}.
In fact, one can notice that this constant was implicitly present in the proof of Theorem~\ref{th-f} where it was shown, in particular, that
$$|\nabla{f}|_{\rho}(x,y)\le |\sd{f}|_{\rho}'(x,y)\quad\mbox{for all}\quad \rho>0.$$
It is easy to check the relationships between constants \eqref{sds-f} and \eqref{sds-f'}:
\begin{enumerate}
\item
$|\sd{f}|_{\rho'}'(x,y)\le|\sd{f}|_{\rho}(x,y)+\rho/\rho'$ for all $\rho>0$ and $\rho'>0$;
\item
if $|\sd{f}|_{\rho}'(x,y)<\ga<\infty$, then $|\sd{f}|_{\ga\rho}(x,y)<\ga$.
\end{enumerate}
An advantage of constant \eqref{sds-f} is that it does not depend on the choice of an equivalent norm in the product space.
\end{remark}

In the special case of function $\tilde f$ defined by \eqref{fti}, slopes \eqref{nls-f}, \eqref{uss-f}, \eqref{ls-f}, \eqref{ss-f}, \eqref{sds-f}, and \eqref{ssds-f} reduce, respectively, to \eqref{gli0}, \eqref{sl0}, \eqref{li}, \eqref{sl-o}, \eqref{eg}, and \eqref{if}.

\section{Error Bounds Criteria for Functions of Two Variables}\label{EB2}

In this section, we establish primal and dual space characterizations of the error bound property \eqref{eb} for a function $f:X\times Y\to\R_\infty$ defined on the product of metric spaces $X$ and $Y$ and satisfying conditions (P1) and (P2).
The criteria formulated below generalize the corresponding ones in Section~\ref{S1}.

\subsection{Nonlocal slope characterization}
The main result is given by the next theorem being an extension of Theorem~\ref{ a}.

\begin{theorem}\label{3T1}
\begin{enumerate}
\item
$\Er f(\bx,\by)\le
\overline{|\nabla{f}|}{}^{\diamond}(\bar{x},\by)$;
\item
if $X$ and $Y$ are complete and $f_+$ is lower semicontinuous (in the product topology)
near $(\bar{x},\by)$,
then
$\Er f(\bx,\by)=
\overline{|\nabla{f}|}{}^{\diamond}(\bar{x},\by)$.
\end{enumerate}
\end{theorem}

\begin{proof}
(i) If $\Er f(\bx,\by)=0$, the
assertion is trivial.
Let $0<\gamma<\Er f(\bx,\by)$.
We are going to show that $\overline{|\nabla{f}|}{}^{\diamond}(\bar{x},\by)\ge\gamma$.
By \eqref{rr2}, there is a $\delta>0$ such that
\begin{equation}\label{3ar}
\frac{f(x,y)}{d(x,S(f))}>\gamma
\end{equation}
for any $x\in{B}_\delta(\bar{x})$ and $y\in Y$ with $f(x,y)>0$.
At the same time, by (P2),
taking a smaller $\de$ if necessary, we can ensure that
\begin{equation}\label{3ar2}
\frac{f(x,y)}{d(y,\by)}>\de\gamma
\end{equation}
for all $(x,y)\in X\times Y$ such that $0<f(x,y)<\de$.
If $\rho\in(0,\de)$, $d(x,\bx)<\rho$, and $0<f(x,y)<\rho$, then,
by \eqref{3ar}, one can find a $u\in S(f)$ such that
$$
\frac{f(x,y)}{d(x,u)}>\gamma.
$$
Taking into account \eqref{3ar2}, we have
\begin{gather*}
\begin{split}
|\nabla{f}|_{\rho}^{\diamond}(x,y)\ge \frac{f(x,y)-f_+(u,\by)}{d_\rho((u,\by),(x,y))}&= \frac{f(x,y)}{d_\rho((u,\by),(x,y))}\\
&=\min\left\{\frac{f(x,y)}{d(x,u)},\frac{f(x,y)}{\rho d(y,\by)}\right\}> \gamma
\end{split}
\end{gather*}
and consequently
$\overline{|\nabla{f}|}{}^{\diamond}(\bar{x},\by)\ge\gamma$.
The claimed inequality follows after letting $\ga\to\Er f(\bx,\by)$.

(ii) Let $X$ and $Y$ be complete and $f_+$ be lower semicontinuous near $(\bar{x},\by)$ (in the product topology).
Thanks to (i), we only need to show that
$$\Er f(\bx,\by)\ge
\overline{|\nabla{f}|}{}^{\diamond}(\bar{x},\by).$$
If $\Er f(\bx,\by)=\infty$, the inequality is trivial.
Let $\Er f(\bx,\by)<\gamma<\infty$.
Choose a $\ga'\in(\Er f(\bx,\by),\gamma)$, a $\delta>0$ such that $f_+$ is lower semicontinuous on $B_\de(\bx,\by)$, a $\be>0$ such that
\begin{gather}\label{3ose}
d(y,\by)<\de/2\quad\mbox{if}\quad 0<f(x,y)<\be,
\end{gather}
a $\rho\in(0,1)$, and set
\begin{gather}\label{3ch-1}
\eta:=\min\{\rho/2,\rho\ga\iv,\de/4,\be\ga\iv\}.
\end{gather}

By \eqref{rr2}, there is a $z\in B_{\eta}(\bx)$ and a $w\in Y$ such that
\begin{gather}\label{3ch0}
0<f(z,w)<\ga' d(z,S(f)).
\end{gather}
Denote $\eps:=f(z,w)$ and $\mu:=d(z,S(f))$.
Then, $\mu\le d(z,\bx)\le\eta$.
Now we consider a complete metric space $(B_\de(\bx,\by),d_\rho)$, where metric $d_\rho$ is defined by \eqref{drho}.
Applying to $f_+$ the Ekeland variational principle (Lemma~\ref{l01}) with $\eps>0$ defined above and
\begin{gather}\label{3ch1-}
\la:=\frac{\ga'}{\ga}\mu,
\end{gather}
we find a point $(x,y)\in B_\de(\bx,\by)$ such that
\begin{gather}\label{3ch1}
d_\rho((x,y),(z,w))\le\la,\quad
f_+(x,y)\le f(z,w),
\end{gather}
and
\begin{gather}\label{3ch2}
f_+(u,v) +(\eps/\la)d_\rho((u,v),(x,y))\ge f_+(x,y),\quad\forall(u,v)\in B_\de(\bx,\by).
\end{gather}

Thanks to \eqref{3ch1}, \eqref{3ch1-}, \eqref{3ch-1}, and \eqref{3ch0}, we have
\begin{gather}
d(x,z)\le\la<\mu\le d(z,\bx),
\notag\\
d(x,S(f))\ge d(z,S(f))-d(x,z)\ge\mu-\la >0,
\label{3ch3}\\
d(x,\bx)\le d(x,z)+d(z,\bx)< 2d(z,\bx)\le2\eta
\le\min\{\rho,\de/2\},
\label{3ch4}\\
f_+(x,y)\le f(z,w)<\ga \mu\le\ga d(z,\bx)\le\ga\eta
\le\min\{\rho,\be\}.
\label{3ch5}
\end{gather}
It follows from \eqref{3ch3} that $f(x,y)>0$, while \eqref{3ch4} and \eqref{3ch5} together with \eqref{3ose} guarantee that
$d(x,\bx)<\rho$,  $f(x,y)<\rho$, and
\begin{gather}\label{3low}
d((x,y),(\bx,\by))<
\de/2.
\end{gather}

Thanks to \eqref{3ch2}, \eqref{3ch1-}, and \eqref{3ch0}, we have
\begin{gather*}
f(x,y)-f_+(u,v)\le \ga d_\rho((u,v),(x,y)), \quad\forall(u,v)\in B_\de(\bx,\by).
\end{gather*}
If $(u,v)\notin B_\de(\bx,\by)$, then, by \eqref{3low},
$$d((u,v),(x,y))>\de-d((x,y),(\bx,\by))>\de/2$$
and consequently, by \eqref{3ch1}, \eqref{3ch0}, \eqref{3ch-1},
\begin{gather*}
\begin{split}
f(x,y)-f_+(u,v)&\le f(x,y)\le f(z,w)<\ga d(z,\bx)\le\ga\eta\le\ga\rho\de/2\\
&<\ga\rho d((u,v),(x,y))\le\ga d_\rho((u,v),(x,y)).
\end{split}
\end{gather*}
Hence,
\begin{gather*}
f(x,y)-f_+(u,v)\le\ga d_\rho((u,v),(x,y)), \quad\forall(u,v)\in X\times Y,
\end{gather*}
or equivalently,
\begin{gather*}
|\nabla{f}|_{\rho}^{\diamond}(x,y) =\sup_{(u,v)\ne(x,y)}
\frac{f(x,y)-f_+(u,v)}{d_\rho((u,v),(x,y))}\le \ga
\end{gather*}
and consequently,
\begin{gather*}
\inf_{\substack{d(x,\bx)<\rho,\,f(x,y)<\rho}}\,
|\nabla{f}|_{\rho}^{\diamond}(x,y)\le \ga.
\end{gather*}
Taking limits in the last inequality as $\rho\downarrow0$ and $\ga\to\Er f(\bx,\by)$ completes the proof.
\end{proof}

It follows from Theorem~\ref{3T1}  that inequality
$\overline{|\nabla{f}|}{}^{\diamond}(\bar{x},\by)>0$ is crucial for determining the error bound property of $f$ at $(\bx,\by)$.

The nonlocal $\rho$-slope \eqref{nls-f} depends on the choice of $\rho$-metric on the product space.
If instead of the maximum type metric $d_\rho$, defined by \eqref{drho}, one employs in \eqref{nls-f} the sum type metric $d_\rho^1$, defined by \eqref{drho1},
it will produce a different number.
We say that a $\rho$-metric $d'_\rho$ on $X\times Y$ is admissible if $d_\rho\le d'_\rho\le d^1_\rho$.
Fortunately, Theorem~\ref{3T1} is invariant on the choice of an admissible metric.

\begin{proposition}\label{3P1}
Theorem~\ref{3T1} remains valid if, in definition \eqref{nls-f}, metric \eqref{drho} is replaced by some other admissible $\rho$-metric.
\end{proposition}

\begin{proof}
Denote by $\overline{|\nabla{f}|}{}^{\diamond}_{1}(\bar{x},\by)$ the constant produced by \eqref{uss-f} if metric \eqref{drho} is replaced in \eqref{nls-f} by metric \eqref{drho1}.
Since a larger metric leads to a smaller value of \eqref{nls-f} and consequently of \eqref{uss-f}, it holds $\overline{|\nabla{f}|}{}^{\diamond}_{1}(\bar{x},\by)\le \overline{|\nabla{f}|}{}^{\diamond}(\bar{x},\by)$ with the constants corresponding to any other admissible $\rho$-metric lying in between.
We only need to prove that $\Er f(\bx,\by)\le
\overline{|\nabla{f}|}{}^{\diamond}_{1}(\bar{x},\by)$.

If $\Er f(\bx,\by)=0$ or
$\overline{|\nabla{f}|}{}^{\diamond}_{1}(\bar{x},\by)=\infty$, the inequality is trivial.
Let $0<\gamma<\Er f(\bx,\by)$ and
$\overline{|\nabla{f}|}{}^{\diamond}_{1}(\bar{x},\by)<\infty$.
We are going to show that $\overline{|\nabla{f}|}{}^{\diamond}_{1}(\bar{x},\by)\ge\gamma$.
Choose a $\ga'\in(\ga,\Er f(\bx,\by))$.
By \eqref{rr2}, there is a $\delta>0$ such that
\begin{equation}\label{3arr}
\frac{d(x,S(f))}{f(x,y)}<(\gamma')\iv
\end{equation}
for any $x\in{B}_\delta(\bar{x})$ and $y\in Y$ with $f(x,y)>0$.
Thanks to (P2), taking a smaller $\de$ if necessary, we can ensure that
\begin{equation}\label{3arr2}
\frac{d(y,\by)}{f(x,y)}<\frac{\gamma\iv-(\ga')\iv}{\de}
\end{equation}
for all $(x,y)\in X\times Y$ such that $0<f(x,y)<\de$.

Choose any $\rho\in(0,\de)$ and any
$(x,y)\in X\times Y$ such that $d(x,\bx)<\rho$ and $0<f(x,y)<\rho$
(Such points exist since
$\overline{|\nabla{f}|}{}^{\diamond}_{1}(\bar{x},\by)<\infty$.) By \eqref{3arr}, one can find a
$u\in S(f)$ such that
$$
\frac{d(x,u)}{f(x,y)}<(\gamma')\iv.
$$
Taking into account \eqref{3arr2}, we also have
\begin{multline*}
\frac{d_\rho^1((u,\by),(x,y))}{f(x,y)-f_+(u,\by)}= \frac{d_\rho^1((u,\by),(x,y))}{f(x,y)}= \frac{d(u,x)+\rho d(y,\by)}{f(x,y)}\\
<(\gamma')\iv+\frac{\rho}{\de}(\gamma\iv-(\ga')\iv)
<\gamma\iv.
\end{multline*}
Hence,
\begin{gather*}
\sup_{(u,v)\ne(x,y)}
\frac{f(x,y)-f_+(u,v)}{d_\rho^1((u,v),(x,y))}\ge
\frac{f(x,y)-f_+(u,\by)}{d_\rho^1((u,\by),(x,y))}>\ga,
\\
\inf_{\substack{d(x,\bx)<\rho,\,0<f(x,y)<\rho}}\,
\sup_{(u,v)\ne(x,y)}
\frac{f(x,y)-f_+(u,v)}{d_\rho^1((u,v),(x,y))}\ge\ga.
\end{gather*}
The claimed inequality follows after taking limits as $\rho\downarrow0$ and $\ga\to\Er f(\bx,\by)$.
\end{proof}

\begin{remark}\label{3R1}
It follows from Theorem~\ref{3T1} and Proposition~\ref{3P1} that, when $X$ and $Y$ are complete and $f_+$ is \lsc\ near $(\bar{x},\by)$, the uniform strict slope \eqref{uss-f} of $f$ at $(\bx,\by)$ is invariant on the choice of an admissible $\rho$-metric on $X\times Y$.
\end{remark}

\subsection{Error bound criteria}
Using slopes \eqref{uss-f}, \eqref{ss-f}, and \eqref{ssds-f}, one can formulate several quantitative criteria of error bounds.
The next corollary is a consequence of Theorems~\ref{3T1} and \ref{th-f} and Proposition~\ref{pr2}.

\begin{corollary}\label{3C1.1}
Let $\ga>0$.
Consider the following conditions:
\renewcommand {\theenumi} {\alph{enumi}}
\begin{enumerate}
\item
$f$ has an error bound at $(\bx,\by)$ with some $\tau>0$;
\item
$\overline{|\nabla{f}|}{}^{\diamond}(\bar{x},\by)>\ga$,\\ i.e., for some $\rho>0$ and any $(x,y)\in X\times Y$ with $d(x,\bx)<\rho$, and $0<f(x,y)<\rho$, it holds $|\nabla{f}|_\rho^{\diamond}(x,y)>\ga$, and consequently there is a $(u,v)\in X\times Y$ such that
\begin{gather*}\label{3ree}
f(x,y)-f_+(u,v)>\ga d_\rho((u,v),(x,y));
\end{gather*}
\item
$\ds\liminf_{x\to\bar{x},\;f(x,y)\downarrow0} \frac{f(x,y)}{d(x,\bx)}>\ga$;
\item
$\overline{|\nabla{f}|}{}^{>}(\bar{x},\by)>\ga$,\\ i.e., for some $\rho>0$ and any $(x,y)\in X\times Y$ with $d(x,\bx)<\rho$ and $0<f(x,y)<\rho$, it holds $|\nabla{f}|_\rho(x,y)>\ga$ and consequently, for any $\eps>0$, there is a $(u,v)\in B_\eps(x,y)$ such that
\begin{gather}\label{3ree2}
f(x,y)-f(u,v)>\ga d_\rho((u,v),(x,y));
\end{gather}
\item
$\ds\liminf_{x\to\bar{x},\;f(x,y)\downarrow0}
\max\left\{|\nabla{f}|(x,y), \frac{f(x,y)}{d(x,\bx)}\right\}>\ga$,\\ i.e., for some $\rho>0$ and any $(x,y)\in X\times Y$ with $d(x,\bx)<\rho$, $0<f(x,y)<\rho$, and $f(x,y)/d(x,\bx)\le\ga$, it holds $|\nabla{f}|_\rho(x,y)>\ga$ and consequently, for any $\eps>0$, there is a $(u,v)\in B_\eps(x,y)$ such that \eqref{3ree2} holds true;
\item
$X$ and $Y$ are normed spaces and $\overline{|\sd{f}|}{}^{>}(\bar{x},\by)>\ga$,\\ i.e.,
for some $\rho>0$ and any $(x,y)\in X\times Y$ with $\|x-\bx\|<\rho$ and $0<f(x,y)<\rho$, it holds $|\sd{f}|_\rho(x,y)>\ga$ and consequently $\|x^*\|>\ga$ for all $(x^*,y^*)\in\sd f(x,y)$ with $\|y^*\|<\rho$.
\item
$X$ and $Y$ are normed spaces and\\ $\ds\liminf_{x\to\bar{x},\;f(x,y)\downarrow0}
\max\left\{|\sd{f}|(x,y), \frac{f(x,y)}{\|x-\bx\|}\right\}>\ga$,\\ i.e.,
for some $\rho>0$ and any $(x,y)\in X\times Y$ with $\|x-\bx\|<\rho$, $0<f(x,y)<\rho$, and $f(x,y)/\|x-\bx\|\le\ga$, it holds $|\sd{f}|_\rho(x,y)>\ga$ and consequently $\|x^*\|>\ga$ for all $(x^*,y^*)\in\sd f(x,y)$ with $\|y^*\|<\rho$.
\end{enumerate}
\renewcommand {\theenumi} {\roman{enumi}}
The following implications hold true:
\begin{enumerate}
\item
{\rm (c) \folgt (e)};
\item
{\rm (d) \folgt (e)};
\item
{\rm (e) \folgt (b)};
\item
if $\ga<\tau$, then {\rm (a) \folgt (b)};
\item
if $X$ and $Y$ are normed spaces, then {\rm (d) \folgt (f)} and {\rm (e) \folgt (g)}.
\cnta
\end{enumerate}
Suppose $X$ and $Y$ are complete and $f_+$ is lower semicontinuous (in the product topology) near $(\bar{x},\by)$.
Then,
\begin{enumerate}
\cntb
\item
if $\tau\le\ga$, then {\rm (b) \folgt (a)}.
\item
if $X$ and $Y$ are Asplund spaces, then
{\rm (d)~\iff (f)} and {\rm (e) \iff (g)}.
\end{enumerate}
\end{corollary}

The next corollary presents a qualitative version of Corollary~\ref{3C1.1}.

\begin{corollary}\label{C1-0}
Suppose $X$ and $Y$ are complete metric spaces and $f_+$ is lower semicontinuous (in the product topology) near $(\bar{x},\by)$.
Then,
$f$ has an error bound at $(\bx,\by)$ provided that one of the following conditions holds true:
\renewcommand {\theenumi} {\alph{enumi}}
\begin{enumerate}
\item
$\overline{|\nabla{f}|}{}^{\diamond}(\bar{x},\by)>0$; \item
$\ds\liminf_{x\to\bar{x},\;f(x,y)\downarrow0} \frac{f(x,y)}{d(x,\bx)}>0$;
\item
$\overline{|\nabla{f}|}{}^{>}(\bar{x},\by)>0$;
\item
$\ds\liminf_{x\to\bar{x},\;\frac{f(x,y)}{d(x,\bx)}\downarrow0}
|\nabla{f}|(x,y)>0$;
\item
$X$ and $Y$ are Asplund spaces and $\overline{|\sd{f}|}{}^{>}(\bar{x},\by)>0$; \item
$X$ and $Y$ are Asplund spaces and $\ds\liminf_{x\to\bar{x},\;\frac{f(x,y)}{\|x-\bx\|}\downarrow0}
|\sd{f}|(x,y)>0$.
\end{enumerate}
\renewcommand {\theenumi} {\roman{enumi}}
Moreover,
\begin{enumerate}
\item
condition {\rm (a)} is also necessary for the local error bound property of $f$ at $(\bx,\by)$;
\item
{\rm (b) \folgt (d)};
\item
{\rm (c) \folgt (d)};
\item
{\rm (d) \folgt (a)};
\item
{\rm (e) \folgt (f)};
\item
if $X$ and $Y$ are Asplund, then {\rm (e) \iff (c)} and {\rm (f) \iff (d)}.
\end{enumerate}
\end{corollary}

\section{Metric subregularity}\label{S3}

From now on, $F:X\rightrightarrows Y$ is a set-valued mapping between metric spaces and $(\bx,\by)\in\gph F$.
We are targeting the metric subregularity property, the main tool being the error bound criteria discussed in the previous section.

\subsection{Definition}
Set-valued mapping $F$ is \emph{metrically subregular} at $(\bx,\by)$ with constant $\tau>0$ if there exists a neighbourhood $U$ of $\bx$ such that
\begin{equation}\label{MR0}
\tau d(x,F^{-1}(\by))\le d(\by,F(x)) \quad \mbox{ for all } x \in U.
\end{equation}

The following (possibly infinite) constant is convenient for characterizing the metric subregularity property:
\begin{gather}\label{CMR}
\sr[F](\bx,\by):=
\liminf_{\substack{x\to\bx\\x\notin F\iv(\by)}}
\frac{d(\by,F(x))}
{d(x,F^{-1}(\by))}.
\end{gather}
It is easy to check that $F$ is metrically subregular at
$(\bx,\by)$ if and only if $\sr[F](\bx,\by)>0$.
Moreover, when positive, constant \eqref{CMR} provides a
quantitative characterization of this property.
It coincides with the supremum of all positive $\tau$ such that \eqref{MR0} holds for some $U$.

Property \eqref{MR0} can be considered as a special case of the error bound property \eqref{eb} while constant \eqref{CMR} reduces to \eqref{rr2} if $f$ is defined on $X\times Y$ by
\begin{gather}\label{f}
f(x,y):=
\begin{cases}
d(y,\by) & \text{if } (x,y)\in\gph F,\\
+\infty & \text{otherwise}.
\end{cases}
\end{gather}
Observe that $f$ is nonnegative, conditions (P1) and (P2) are trivially satisfied, and
\begin{gather}\label{S(F)}
S(f)=F^{-1}(\by).
\end{gather}

Another important observation is that besides \eqref{f} one can relate to $F$ other real-valued functions satisfying conditions (P1), (P2), and \eqref{S(F)}.
This way, it is possible to generalize the criteria presented in the rest of the paper to nonlinear, particularly H\"older-type, regularity properties.

\subsection{Primal space slopes}
The nonlocal slopes \eqref{nls-f} and \eqref{uss-f} of $f$ in the current setting take the following form:
\begin{gather*}\label{7nls}
|\nabla{F}|_{\rho}^{\diamond}(x,y):=
\sup_{\substack{(u,v)\ne(x,y)\\(u,v)\in\gph F}}
\frac{[d(y,\by)-d(v,\by)]_+}{d_\rho((u,v),(x,y))},
\\\label{7uss}
\overline{|\nabla{F}|}{}^{\diamond}(\bar{x},\by):=
\lim_{\rho\downarrow0}
\inf_{\substack{d(x,\bx)<\rho,\,d(y,\by)<\rho\\
(x,y)\in\gph F,\,x\notin F\iv(\by)}}\,
|\nabla{F}|{}^{\diamond}_{\rho}(x,y).
\end{gather*}
We will call the above constants, respectively,
the \emph{nonlocal $\rho$-slope} of $F$ at $(x,y)\in\gph F$ and
the \emph{uniform strict slope} of $F$ at $(\bx,\by)$.

The local slopes \eqref{ls-f} and \eqref{ss-f}, when applied to function \eqref{f}, produce the following definitions:
\begin{align}\label{srho}
|\nabla{F}|_{\rho}(x,y):=&
\limsup_{\substack{u\to x,\,v\to y,\,
(u,v)\ne(x,y)\\
(u,v)\in\gph F}}
\frac{[d(y,\by)-d(v,\by)]_+} {d_\rho((u,v),(x,y))},
\\\label{q-ss}
\overline{|\nabla{F}|}{}(\bar{x},\by):=&
\lim_{\rho\downarrow0}
\inf_{\substack{d(x,\bx)<\rho,\,d(y,\by)<\rho\\
(x,y)\in\gph F,\,x\notin F\iv(\by)}}\,
|\nabla{F}|_{\rho}(x,y).
\end{align}
They are called, respectively,
the (local) \emph{$\rho$-slope} of $F$ at $(x,y)\in\gph F$ and
the \emph{strict slope} of $F$ at
$(\bx,\by)$.

The next statement is a consequence of Proposition~\ref{pr2}.

\begin{proposition}\label{P2}
\begin{enumerate}
\item
$|\nabla{F}|_{\rho}(x,y)\le
|\nabla{F}|_{\rho}^{\diamond}(x,y)$
for all $\rho>0$ and $(x,y)\in\gph F$,
\item
$\overline{|\nabla{F}|}{}(\bar{x},\by)\le
\overline{|\nabla{F}|}{}^{\diamond}(\bx,\by)$,
\item
$\overline{|\nabla{F}|}{}^{\diamond}(\bar{x},\by)\ge
\ds\liminf_{\substack{x\to\bar{x},\;y\to\by\\
(x,y)\in\gph F,\,x\notin F\iv(\by)}}
\frac{d(y,\by)}{d(x,\bx)}$.
\end{enumerate}
\end{proposition}

\subsection{Subdifferential slopes}
If $X$ and $Y$ are normed linear spaces, one can define the \emph{subdifferential $\rho$-slope} ($\rho>0$) of $F$ at $(x,y)\in\gph F$ with $y\ne\by$ as
\begin{gather}\label{6srs}
|\sd{F}|_{\rho}(x,y)
:=\inf_{\substack{x^*\in D^*F(x,y)(J(y-\by)+\rho\B^*)}}
\|x^*\|,
\end{gather}
where $J$ is the duality mapping defined by \eqref{J}.

Using \eqref{6srs}, we define the \emph{strict subdifferential slope} of $F$ at $(\bx,\by)$:
\begin{align}\label{phi-sss}
\overline{|\sd{F}|}{}(\bar{x},\by):=&
\lim_{\rho\downarrow0}
\inf_{\substack{\|x-\bx\|<\rho,\,\|y-\by\|<\rho\\
(x,y)\in\gph F,\,x\notin F\iv(\by)}}\,
|\sd{F}|_{\rho}(x,y).
\end{align}

In some situations, more advanced versions of \eqref{6srs} and \eqref{phi-sss} are required:
\begin{align}\label{6asrs}
|\sd{F}|^a_{\rho}(x,y)
:=&\liminf_{\substack{v\to y-\by}}\
\inf_{\substack{x^*\in D^*F(x,y)(J(v)+\rho\B^*)}}
\|x^*\|,
\\\label{phi-asss}
\overline{|\sd{F}|}{}^a(\bar{x},\by):=&
\lim_{\rho\downarrow0}
\inf_{\substack{\|x-\bx\|<\rho,\,\|y-\by\|<\rho\\
(x,y)\in\gph F,\,x\notin F\iv(\by)}}\,
|\sd{F}|^a_{\rho}(x,y).
\end{align}
They are called, respectively,
the \emph{approximate subdifferential $\rho$-slope} ($\rho>0$) of $F$ at $(x,y)\in\gph F$ with $y\ne\by$ and the \emph{approximate strict subdifferential slope} of $F$ at $(\bx,\by)$.

The next proposition gives relationships between the subdifferential slopes \eqref{6srs}--\eqref{phi-asss} which
follow directly from the definitions.

\begin{proposition}\label{5P31}
\begin{enumerate}
\item
$|\sd{F}|^a_{\rho}(x,y) \le|\sd{F}|_{\rho}(x,y)$\\ for all $\rho>0$ and $(x,y)\in\gph F$;
\item
$\overline{|\sd{F}|}{}^a(\bar{x},\by) \le\overline{|\sd{F}|}{}(\bar{x},\by)$.
\end{enumerate}
\end{proposition}

The next proposition establishes relationships between the approximate subdifferential slopes \eqref{6asrs} and \eqref{phi-asss} of set-valued mapping $F$ and the corresponding ones of function $f$ defined by \eqref{f} in the Asplund space setting.

\begin{proposition}\label{5P32}
Suppose $X$ and $Y$ are Asplund, $\gph F$ is locally closed near $(\bar{x},\by)$, and function $f$ is given by \eqref{f}.
Then,
\begin{enumerate}
\item
$\ds|\sd{f}|_{\rho}(x,y)
\ge \liminf_{\substack{ (x',y')\to(x,y)\\(x',y')\in\gph F}}\
|\sd{F}|^a_{\rho}(x',y')$\\ for all $\rho>0$ and $(x,y)\in\gph F$ near $(\bar{x},\by)$ with $y\ne\by$;
\item
$\ds\overline{|\sd{f}|}{}^{>}(\bar{x},\by) \ge\overline{|\sd{F}|}{}^a(\bar{x},\by)$.
\end{enumerate}
\end{proposition}

\begin{proof}
(i) Let $\rho>0$ and $(x,y)\in\gph F$ near $(\bar{x},\by)$ with $y\ne\by$ be given such that $\gph F$ is locally closed near $(x,y)$.
Observe that function $f$ is the sum of two functions on $X\times Y$:
$$(u,v)\mapsto \|v-\by\| \quad\mbox{and}\quad (u,v) \mapsto \delta_{\gph F}(u,v),$$
where $\delta_{\gph F}$ is the \emph{indicator function} of $\gph F$: $\delta_{\gph F}(u,v)=0$ if $(u,v)\in\gph F$ and $\delta_{\gph F}(u,v)=\infty$ otherwise.
Considering $X\times Y$ with the product topology, by Lemmas~\ref{l02} and \ref{ll01},
for any $\eps>0$, it holds
\begin{gather*}\label{5f01}
\sd f(x,y)
\subset\bigcup_{\substack{
\|(x',y')-(x,y)\|<\eps,\
(x',y')\in\gph F\\
(x^*,y^*)\in N_{\gph F}(x',y')\\
\|y''-y\|<\eps,\,v^*\in J(y'')}} \{x^*,y^*+v^*\}+\eps\B_{X^*\times Y^*}.
\end{gather*}
By definition \eqref{sds-f},
\begin{align*}
|\sd{f}|_{\rho}(x,y)
&\ge\inf_{\substack{
\|(x',y')-(x,y)\|<\eps,\
(x',y')\in\gph F\\
(x^*,y^*)\in N_{\gph F}(x',y')\\
\|y''-y\|<\eps,\,v^*\in J(y'')\\ \|y^*+v^*\|<\rho}} \|x^*\|-\eps
\\
&=\inf_{\substack{
\|(x',y')-(x,y)\|<\eps,\
(x',y')\in\gph F\\
x^*\in D^*F(x',y')(y^*)\\
\|y''-y\|<\eps,\,v^*\in J(y'')\\ \|y^*-v^*\|<\rho}} \|x^*\|-\eps
\\
&=\inf_{\substack{
\|(x',y')-(x,y)\|<\eps,\
(x',y')\in\gph F\\
x^*\in D^*F(x',y')(J(y'')+\rho\B^*)\\
\|y''-y\|<\eps}} \|x^*\|-\eps.
\end{align*}
Hence, by definition \eqref{6asrs},
\begin{align*}
|\sd{f}|_{\rho}(x,y)
&\ge\inf_{\substack{
\|(x',y')-(x,y)\|<\eps\\
(x',y')\in\gph F}} |\sd{F}|^a_{\rho}(x',y')-\eps.
\end{align*}
The conclusion follows after passing to the limit in the \RHS\ of the above inequality as $\eps\downarrow0$.

(ii) By (i) and definition \eqref{phi-asss},
for any $\eps>0$, we have:
\begin{align*}
\overline{|\sd{f}|}{}^{>}(\bar{x},\by)
&\ge\lim_{\rho\downarrow0}
&&\inf_{\substack{\|x-\bx\|<\rho,\,\|y-\by\|<\rho\\
(x,y)\in\gph F,\,
x\notin F\iv(\by)}}\,
\liminf_{\substack{ (x',y')\to(x,y)\\(x',y')\in\gph F}}\
|\sd{F}|{}^a_{\rho}(x',y')
\\
&\ge\lim_{\rho\downarrow0}
&&\inf_{\substack{\|x-\bx\|<\rho,\,\|y-\by\|<\rho\\
(x,y)\in\gph F,\,
x\notin F\iv(\by)}}\,
\inf_{\substack{\|x'-x\|<\eps,\,\|y'-y\|<\eps\\ (x',y')\in\gph F}}\
|\sd{F}|^a_{\rho}(x',y').
\end{align*}
Choosing, for a fixed $(x,y)$, a sufficiently small positive $\eps<\rho-\max\{\|x-\bx\|,\|y-\by\|\}$, we can ensure that $B_\eps(x)\cap F\iv(\by)=\emptyset$. Hence,
\begin{align*}
\overline{|\sd{f}|}{}^{>}(\bar{x},\by)
&\ge\lim_{\rho\downarrow0}
\inf_{\substack{\|x'-\bx\|<\rho,\,\|y'-\by\|<\rho\\
(x',y')\in\gph F,\,
x'\notin F\iv(\by)}}\,
|\sd{F}|_{\rho}(x',y') =\overline{|\sd{F}|}{}^a(\bar{x},\by).
\end{align*}
\end{proof}

Note that, unlike the primal space local slopes \eqref{srho} and \eqref{q-ss}, the approximate subdifferential slopes \eqref{6asrs} and \eqref{phi-asss} are not in general exact realizations of the corresponding subdifferential slopes \eqref{sds-f} and \eqref{ssds-f} when applied to function \eqref{f}.
Proposition~\ref{5P32} guarantees only inequalities and only in the Asplund space setting, the main tool being the fuzzy sum rule (Lemma~\ref{l02}) valid in Asplund spaces.
The next proposition presents
an important case of equalities in general normed spaces involving simpler
subdifferential slopes \eqref{6srs} and \eqref{phi-sss}.
The proof is similar to that of Proposition~\ref{5P32} with the replacement of the fuzzy sum rule by the exact either differentiable rule (see, e.g., \cite[Corollary~1.12.2]{Kru03.1}) or the convex sum rule (Moreau--Rockafellar formula).

\begin{proposition}\label{5P33}
If $X$ and $Y$ are normed spaces and either the norm in $Y$ is Fr\'echet differentiable away from $0_Y$, or $F$ is convex, then
\begin{enumerate}
\item
$\ds|\sd{f}|_{\rho}(x,y)
= |\sd{F}|_{\rho}(x,y)$\\ for all $\rho>0$ and $(x,y)\in\gph F$ near $(\bar{x},\by)$ with $y\ne\by$;
\item
$\ds\overline{|\sd{f}|}{}^{>}(\bar{x},\by) =\overline{|\sd{F}|}{}(\bar{x},\by)$.
\end{enumerate}
\end{proposition}

The next proposition gives a relationship between the primal space strict slope \eqref{q-ss} and the approximate strict subdifferential slope \eqref{phi-asss} of a set-valued mapping $F$ in the Asplund space setting.
It is a consequence of Theorem~\ref{th-f} and Proposition~\ref{5P32}.

\begin{proposition}\label{5P2}
If $X$ and $Y$ are Asplund and $\gph F$ is locally closed near $(\bar{x},\by)$, then
$\overline{|\nabla{F}|}{}(\bar{x},\by)\ge
\overline{|\sd{F}|}{}^{a}(\bx,\by)$.
\end{proposition}

\subsection{Criteria of metric subregularity}
We first get back to the original setting of a set-valued mapping $F:X\rightrightarrows Y$ between metric spaces with $(\bx,\by)\in\gph F$.
The next theorem is a consequence of Theorem~\ref{3T1}.

\begin{theorem}\label{7T1}
\begin{enumerate}
\item
$\sr[F](\bx,\by)\le
\overline{|\nabla{F}|}{}^{\diamond}(\bar{x},\by)$;
\item
if $X$ and $Y$ are complete and $\gph F$ is locally closed (in the product topology)
near $(\bar{x},\by)$,
then
$\sr[F](\bx,\by)=
\overline{|\nabla{F}|}{}^{\diamond}(\bar{x},\by)$.
\end{enumerate}
\end{theorem}

By Proposition~\ref{3P1}, Theorem~\ref{7T1} is invariant on the choice of an admissible metric on $X\times Y$.

In the convex case, one can formulate a precise estimate in terms of subdifferential slopes in the Banach space setting.

\begin{proposition}\label{iti}
Suppose $X$ and $Y$ are Banach spaces and $\gph F$ is convex and locally closed (in the product topology) near $(\bx,\by)$.
Then, $\sr[F](\bar{x},\by) =\overline{|\sd{F}|}(\bar{x},\by)$.
\end{proposition}

\begin{proof}
Inequality $\sr[F]\ge
\overline{|\sd{F}|}(\bar{x},\by)$
follows from Theorem~\ref{7T1} and Propositions~\ref{P2} and \ref{5P33}.
Next we show that $\sr[F]\le
\overline{|\sd{F}|}(\bar{x},\by)$.
If $\sr[F](\bar{x},\by)=0$, the inequality is trivial.
Suppose $0<\tau<\sr[F](\bar{x},\by)$ and $0<\ga<1$.
Then, by \eqref{CMR}, there exists a $\rho\in(0,1-\ga)$ such that
\begin{gather}\label{iti1}
\tau d(x,F^{-1}(\by))<\|y-\by\|, \quad \forall x\in B_\rho(\bx)\setminus F\iv(\by),\; y\in F(x).
\end{gather}
Choose an arbitrary $(x,y)\in\gph F$ with $\|x-\bx\|<\rho$, $\|y-\by\|<\rho$, $x\notin F\iv(\by)$; $v^*\in J(y-\by)$; and $x^*\in D^*F(x,y)(v^*+\rho\B^*)$. By \eqref{iti1}, one can find a point $u\in F\iv(\by)$ such that
\begin{gather}\label{iti4}
\tau\|x-u\|<\|y-\by\|.
\end{gather}
By the convexity of $F$, the Fr\'echet normal cone to its graph coincides with the normal cone in the sense of convex analysis, and consequently it holds
\begin{gather*}
\langle x^*,u-x\rangle\le\langle v^*,\by-y\rangle +\rho\|y-\by\|=-(1-\rho)\|y-\by\|.
\end{gather*}
Combining this with \eqref{iti4}, we have
\begin{align*}
\|x^*\|\|u-x\| &\ge-\langle x^*,u-x\rangle
\ge(1-\rho)\|y-\by\|
>\ga\|y-\by\| >\ga\tau\|u-x\|.
\end{align*}
Hence,
$\|x^*\|>\ga\tau$,
and it follows from definitions \eqref{phi-sss} and \eqref{6srs} that $\overline{|\sd{F}|}(\bar{x},\by)>\ga\tau$.
Passing to the limit in the last inequality as $\ga\to1$ and $\tau\to\sr[F](\bar{x},\by)$, we arrive at the claimed inequality.
\end{proof}

The next corollary summarizes necessary and sufficient quantitative criteria for metric subregularity.

\begin{corollary}\label{7C1.1}
Let $\ga>0$.
Consider the following conditions:
\renewcommand {\theenumi} {\alph{enumi}}
\begin{enumerate}
\item
$F$ is metrically subregular at $(\bx,\by)$ with some $\tau>0$;
\item
$\overline{|\nabla{F}|}{}^{\diamond}(\bar{x},\by)>\ga$,\\ i.e., for some $\rho>0$ and any $(x,y)\in\gph F$ with $x\notin F\iv(\by)$, $d(x,\bx)<\rho$, and $d(y,\by)<\rho$, it holds $|\nabla{F}|_{\rho}^{\diamond}(x,y)>\ga$, and consequently there is a $(u,v)\in\gph F$ such that
\begin{gather*}\label{7ree}
d(y,\by)-d(v,\by)>\ga d_\rho((u,v),(x,y));
\end{gather*}
\item
$\ds\liminf_{\substack{x\to\bar{x},\;y\to\by\\
(x,y)\in\gph F,\,x\notin F\iv(\by)}}
\frac{d(y,\by)}{d(x,\bx)}>\ga$;
\item
$\overline{|\nabla{F}|}{}(\bar{x},\by)>\ga$,\\ i.e., for some $\rho>0$ and any $(x,y)\in\gph F$ with $x\notin F\iv(\by)$, $d(x,\bx)<\rho$, and $d(y,\by)<\rho$, it holds $|\nabla{F}|_{\rho}(x,y)>\ga$, and consequently, for any $\eps>0$, there is a $(u,v)\in\gph F$ with $d(u,x)<\eps$ and $d(v,y)<\eps$ such that
\begin{gather}\label{7phiree}
d(y,\by)-d(v,\by)>\ga d_\rho((u,v),(x,y));
\end{gather}
\item
$\ds\lim_{\rho\downarrow0}
\inf_{\substack{d(x,\bx)<\rho,\,d(y,\by)<\rho\\
(x,y)\in\gph F,\,x\notin F\iv(\by)}}\,
\max\left\{|\nabla{F}|_\rho(x,y), \frac{d(y,\by)}{d(x,\bx)}\right\}>\ga$,\\ i.e., for some $\rho>0$ and any $(x,y)\in\gph F$ with $x\notin F\iv(\by)$, $d(x,\bx)<\rho$, $d(y,\by)<\rho$, and $d(y,\by)/d(x,\bx)\le\ga$ it holds $|\nabla{F}|_{\rho}(x,y)>\ga$, and consequently, for any $\eps>0$, there is a $(u,v)\in\gph F$ with $d(u,x)<\eps$ and $d(v,y)<\eps$ such that \eqref{7phiree} holds true;
\item\label{35d}
$X$ and $Y$ are normed spaces and $\overline{|\sd{F}|}{}^a(\bar{x},\by)>\ga$,\\ i.e.,
for some $\rho>0$ and any $(x,y)\in\gph F$ with $x\notin F\iv(\by)$, $\|x-\bx\|<\rho$, and $\|y-\by\|<\rho$, it holds $|\sd{F}|^a_\rho(x,y)>\ga$, and consequently there exists an $\eps>0$ such that
\begin{gather}\label{7f}
\|x^*\|>\ga\;\;\mbox{for all }x^*\in D^*F(x,y)(J(B_\eps(y-\by))+\rho\B^*);
\end{gather}
\item
$X$ and $Y$ are normed spaces and\\ $\ds\lim_{\rho\downarrow0}
\inf_{\substack{\|x-\bx\|<\rho,\,\|y-\by\|<\rho\\
(x,y)\in\gph F,\,x\notin F\iv(\by)}}\,
\max\left\{|\sd{F}|{}^a_\rho(x,y), \frac{\|y-\by\|}{\|x-\bx\|}\right\}>\ga$,\\ i.e.,
for some $\rho>0$ and any $(x,y)\in\gph F$ with $x\notin F\iv(\by)$, $\|x-\bx\|<\rho$, $\|y-\by\|<\rho$, and $\|y-\by\|/\|x-\bx\|\le\ga$, it holds $|\sd{F}|^a_\rho(x,y)>\ga$, and consequently there exists an $\eps>0$ such that \eqref{7f} holds true;
\item
$X$ and $Y$ are normed spaces and
$\overline{|\sd{F}|}{}(\bar{x},\by)>\ga$,\\ i.e.,
for some $\rho>0$ and any $(x,y)\in\gph F$ with $x\notin F\iv(\by)$, $\|x-\bx\|<\rho$, and $\|y-\by\|<\rho$, it holds $|\sd{F}|_{\rho}(x,y)>\ga$,
and consequently
\begin{gather*}\label{7f2}
\|x^*\|>\ga\;\;\mbox{for all }x^*\in D^*F(x,y)(J(y-\by)+\rho\B^*).
\end{gather*}
\end{enumerate}
\renewcommand {\theenumi} {\roman{enumi}}
The following implications hold true:
\begin{enumerate}
\item
{\rm (c) \folgt (e)};
\item
{\rm (d) \folgt (e)};
\item
{\rm (e) \folgt (b)};
\item
{\rm \eqref{35d} \folgt (g)};
\item
{\rm \eqref{35d} \folgt (h)};
\item
if $\ga<\tau$, then {\rm (a) \folgt (b)}.
\cnta
\end{enumerate}
Suppose $X$ and $Y$ are complete, $\gph F$ is locally closed (in the product topology) near $(\bx,\by)$ and.
Then,
\begin{enumerate}
\cntb
\item
if $\tau\le\ga$, then {\rm (b) \folgt (a)};
\item
if $X$ and $Y$ are Asplund, then {\rm \eqref{35d}~\folgt (d)} and {\rm (g)~\folgt (e)};
\item
if $X$ and $Y$ are Banach and either the norm of $Y$ is Fr\'echet differentiable away from $0_Y$, or $F$ is convex, then {\rm (h)~\folgt (b)}.
\end{enumerate}
\renewcommand {\theenumi} {\roman{enumi}}
\end{corollary}

Criterion \eqref{35d} in the above proposition generalizes \cite[Proposition~2.2]{IofOut08}, cf. \cite[Theorem~3.1]{ZheNg10}, \cite[Theorem~4.5]{Hua12}, \cite[Theorem~5.1]{ZheNg12}, \cite[Theorem~4.1]{Ngh14},
\cite[Theorem~4.1]{ZheHe14}.

The next corollary presents a qualitative version of Corollary~\ref{7C1.1}.

\begin{corollary}\label{C1-0+}
Suppose $X$ and $Y$ are complete metric spaces and $\gph F$ is locally closed (in the product topology) near $(\bar{x},\by)$.
Then,
$F$ is metrically subregular at $(\bx,\by)$ provided that one of the following conditions holds true:
\renewcommand {\theenumi} {\alph{enumi}}
\begin{enumerate}
\item
$\overline{|\nabla{F}|}{}^{\diamond}(\bar{x},\by)>0$; \item
$\ds\liminf_{\substack{x\to\bar{x},\;y\to\by\\
(x,y)\in\gph F,\,x\notin F\iv(\by)}}
\frac{d(y,\by)}{d(x,\bx)}>0$;
\item
$\overline{|\nabla{F}|}{}(\bar{x},\by)>0$;
\item
$\ds\lim_{\rho\downarrow0}
\inf_{\substack{d(x,\bx)<\rho,\,d(y,\by)<\rho\\
(x,y)\in\gph F,\,x\notin F\iv(\by)}}\,
\max\left\{|\nabla{F}|_\rho(x,y), \frac{d(y,\by)}{d(x,\bx)}\right\}>0$; \item
$X$ and $Y$ are Asplund spaces and $\overline{|\sd{F}|}{}^{a}(\bar{x},\by)>0$;
\item
$X$ and $Y$ are Asplund spaces and \begin{gather}\label{31f}
\ds\lim_{\rho\downarrow0}
\inf_{\substack{\|x-\bx\|<\rho,\,\|y-\by\|<\rho\\
(x,y)\in\gph F,\,x\notin F\iv(\by)}}\,
\max\left\{|\sd{F}|{}^a_\rho(x,y), \frac{\|y-\by\|}{\|x-\bx\|}\right\}>0; \end{gather}
\item
$X$ and $Y$ are Banach spaces, either the norm of $Y$ is Fr\'echet differentiable away from $0_Y$ or $F$ is convex, and $\overline{|\sd{F}|}{}(\bar{x},\by)>0$.
\end{enumerate}
\renewcommand {\theenumi} {\roman{enumi}}
Moreover,
\begin{enumerate}
\item
condition {\rm (a)} is also necessary for the metric subregularity of $F$ at $(\bx,\by)$;
\item
{\rm (b) \folgt (d)};
\item
{\rm (c) \folgt (d)};
\item
{\rm (d) \folgt (a)};
\item
{\rm (e) \folgt (c)};
\item
{\rm (e) \folgt (f)};
\item
{\rm (f) \folgt (d)};
\item
{\rm (g) \folgt (c)}.
\end{enumerate}
\end{corollary}

Criterion (a) in the above corollary (in the more general H\"older setting) can be found in \cite[Proposition~3.4]{Kum09}, see also \cite[Theorem~1]{KlaKum06}.

A sufficient metric subregularity criterion similar to condition (f) was suggested recently by Gfrerer \cite{Gfr11}.
A key ingredient of this criterion is the following \emph{limit set} \cite[Definition 3.1]{Gfr11}:
\begin{align*}
{\rm Cr}_0 F(\bx,\by):=&\{(v,x^*)\in Y\times X^*\mid \exists (t_k)\downarrow0,\ (v_k,x_k^*)\to(v,x^*),\\ &(u_k,y_k^*)\subset\Sp_X\times\Sp_{Y^*} \mbox{ with } x_k^*\in D^*F(\bx+t_ku_k,\by+t_kv_k)(y_k^*)\}.
\end{align*}
The next theorem is the Asplund space part of \cite[Theorem 3.2]{Gfr11}:

\begin{theorem}\label{Gfr}
Suppose $X$ and $Y$ are Asplund and $\gph F$ is locally closed. If $(0,0)\notin{\rm Cr}_0 F(\bx,\by)$, then $F$ is metrically subregular at $(\bx,\by)$.
\end{theorem}

This theorem is a consequence of Corollary~\ref{C1-0+} thanks to the next fact.

\begin{proposition}\label{Gfr1}
If $(0,0)\notin{\rm Cr}_0 F(\bx,\by)$, then condition \eqref{31f} holds true. \end{proposition}

\begin{proof}
Let condition \eqref{31f} fail.
Then, for any $k=1,2,\ldots$, there exists a point $(x_k,y_k)\in\gph F$ such that $x_k\notin F\iv(\by)$, $\|x_k-\bx\|<1/k$,
$\|y_k-\by\|<1/k$,
$\|y_k-\by\|/\|x_k-\bx\|<1/k$, and $|\sd{F}|{}^a_{1/k}(x_k,y_k)<1/k$.
By definition \eqref{6asrs}, there exist elements $v_k^*\in Y^*$ with $\|v_k^*\|>1-1/k$ and $u_k^*\in D^*F(x_k,y_k)(v_k^*)$ with $\|u_k^*\|<1/k$.
Denote $t_k:=\|x_k-\bx\|$, $u_k:=t_k\iv(x_k-\bx)$,
$v_k:=t_k\iv(y_k-\by)$, $y_k^*=v_k^*/\|v_k^*\|$, and
$x_k^*=u_k^*/\|v_k^*\|$.
Obviously, $0<t_k<1/k$, $\|u_k\|=1$, $\|v_k\|<1/k$, $u_k=\bx+t_ku_k$,
$v_k=\by+t_kv_k$,
$\|y_k^*\|=1$, $\|x_k^*\|<(1/k)/(1-1/k) =1/(k-1)$ and $x_k^*\in D^*F(\bx+t_ku_k,\by+t_kv_k)(y_k^*)$.
It holds $t_k\downarrow0$, $v_k\to0$, and $x_k^*\to0$.
Hence, $(0,0)\in{\rm Cr}_0 F(\bx,\by)$.
\end{proof}

The main difference between conditions \eqref{31f} and
$(0,0)\notin{\rm Cr}_0 F(\bx,\by)$, which makes the first one weaker, is the requirement in definition \eqref{6asrs} that $y^*$ component of the pair $(x^*,y^*)\in\gph D^*F(x,y)$ is related to $y-\by$: $y^*\in J(v)+\rho\B^*$ where $v\to y-\by$. (The only requirement in the definition of ${\rm Cr}_0 F(\bx,\by)$ is $y_k^*\in\Sp_{Y^*}$.)
This issue seems to have been taken into account in the most recent publication by Gfrerer \cite[Corollary 1]{Gfr14}.

\begin{remark}
It is easy to see from the proof of Proposition~\ref{Gfr1} that its conclusion remains true if condition $(0,0)\notin{\rm Cr}_0 F(\bx,\by)$ is replaced by a weaker one involving \emph{outer limit set} ${\rm Cr}_0^> F(\bx,\by)$ \cite[p.~1450]{Gfr11}, \cite[p.~156]{Gfr13}.
\end{remark}

\section*{Acknowledgements}
The author is very happy to submit his paper to the special issue dedicated to the 40th Anniversary of the journal where his first paper in English was published in 1988.
The author is grateful to the then editor-in-chief, Professor Karl-Heinz Elster for his support and keeps in his archive a postcard signed by Professor Elster informing the author about the acceptance of that paper.

\section*{Funding}
This work was supported by the Australian Research Council, grant DP110102011.

\bibliographystyle{gOPT}
\bibliography{buch-kr,kruger,kr-tmp}
\end{document}